\documentclass{amsart}

\usepackage{amsfonts}
\usepackage{bm}
\usepackage{bbm}

\usepackage[scr=rsfs]{mathalpha}

\usepackage{bm}

\usepackage{rotfloat}

\usepackage{lscape}
\usepackage{hyperref}
\usepackage{tcolorbox}
\usepackage{subcaption}
\usepackage{graphicx}
\usepackage{enumitem}   

\usepackage{cleveref}
\usepackage{mathtools}
\usepackage{geometry}   
\usepackage{framed,multirow,booktabs}

\usepackage{array}

\usepackage{tikz}
\usetikzlibrary{positioning}
\usetikzlibrary{shapes, arrows}
\tikzstyle{terminator} = [rectangle, draw, text centered, rounded corners, minimum height=2em]
\tikzstyle{process} = [rectangle, draw, text centered, minimum height=2em]
\tikzstyle{decision} = [diamond, draw, text centered, minimum height=2em]
\tikzstyle{data}=[trapezium, draw, text centered, trapezium left angle=60, trapezium right angle=120, minimum height=2em]
\tikzstyle{connector} = [draw, -latex, thick]

\newtheorem{theorem}{Theorem}[section]
\newtheorem{lemma}[theorem]{Lemma}

\newtheorem{corollary}{Corollary}[section]

\theoremstyle{remark}
\newtheorem{rmk}[theorem]{Remark}
\newtheorem{example}[theorem]{\bf Example}

\theoremstyle{problem}

\theoremstyle{definition}

\numberwithin{equation}{section}

\newcommand{\dx}{\Delta x}
\newcommand{\dy}{\Delta y}
\newcommand{\dt}{\Delta t}
\newcommand{\bU}{{\bm u}}
\newcommand{\bF}{{\bm f}}
\newcommand{\bG}{{\bm g}}
\newcommand{\bd}{{\bm d}}
\newcommand{\bx}{{\bm x}}

\newcommand{\G}{\mathcal{G}}

\newcommand*\xbar[1]{%
	\hbox{%
		\vbox{%
			\hrule height 0.5pt 
			\kern0.3ex
			\hbox{%
				\kern-0.05em
				\ensuremath{#1}%
				\kern-0.05em
			}%
		}%
	}%
}

\begin{document}
\global\pdfpageattr\expandafter{\the\pdfpageattr/Rotate 0}

\title[Bound-Preserving Framework for Central-Upwind Schemes]{Bound-Preserving Framework for Central-Upwind Schemes for General Hyperbolic Conservation Laws}

\author{Shumo Cui}
\address{Shenzhen International Center for Mathematics, Southern University of Science and Technology, Shenzhen 518055, China}
\email{cuism@sustech.edu.cn}

\author{Alexander Kurganov}
\address{Department of Mathematics, Shenzhen International Center for Mathematics, and Guangdong Provincial Key Laboratory of Computational Science and Material Design, Southern University of Science and Technology, Shenzhen 518055, China}
\email{alexander@sustech.edu.cn}

\author{Kailiang Wu}
\address{Corresponding author; Department of Mathematics and Shenzhen International Center for Mathematics, Southern University of Science and Technology, and National Center for Applied Mathematics Shenzhen (NCAMS), Shenzhen 518055, China}
\email{wukl@sustech.edu.cn}
\thanks{
	The work of S. Cui was supported in part by Shenzhen Science and Technology Program (grant No. RCJC20221008092757098).
	The work of A. Kurganov was supported in part by NSFC grant 12171226 and by the fund of the Guangdong Provincial Key Laboratory of Computational Science and Material Design (No. 2019B030301001). 
	The work of K. Wu was supported in part by Shenzhen Science and Technology	Program (grant No. RCJC20221008092757098) and NSFC grant 12171227.}





\begin{abstract}
	Central-upwind (CU) schemes are Riemann-problem-solver-free finite-volume methods widely applied to a variety of hyperbolic systems of PDEs.
	Exact solutions of these systems typically satisfy certain bounds, and it is highly desirable or even crucial for the numerical schemes to
	preserve these bounds. In this paper, we develop and analyze bound-preserving (BP) CU schemes for general hyperbolic systems of conservation
	laws. Unlike many other Godunov-type methods, CU schemes cannot, in general, be recast as convex combinations of first-order BP schemes.
	Consequently, standard BP analysis techniques are invalidated. We address these challenges by establishing a novel framework for analyzing
	the BP property of CU schemes. To this end, we discover that the CU schemes can be decomposed as a convex combination of several
	intermediate solution states. Thanks to this key finding, the goal of designing BPCU schemes is simplified to the enforcement of four more
	accessible BP conditions, each of which can be achieved with the help of a minor modification of the CU schemes. We employ the proposed
	approach to construct provably BPCU schemes for the Euler equations of gas dynamics. The robustness and effectiveness of the BPCU schemes
	are validated by several demanding numerical examples, including high-speed jet problems, flow past a forward-facing step, and a shock
	diffraction problem. 
\end{abstract}

\maketitle



\section{Introduction}
This paper focuses on the development of high-order robust numerical schemes for hyperbolic systems of conservation laws  
\begin{equation}
	\bU_t+\nabla_\bx\cdot\bF(\bU)=\bm0, 
	\label{1}
\end{equation}
where $\bx\in\mathbb R^d$ are the spatial variables, $t\in\mathbb R^+$ denotes time, $\bU\in\mathbb R^m(\bx,t)$ is an unknown
vector-function, and $\bF=(f_1,\dots,f_d)^\top\in(\mathbb R^m)^d$ are the fluxes.

Exact solutions of \eqref{1} often satisfy certain bounds that define a convex invariant region ${\G}\in\mathbb R^m$. For instance, the
entropy solution of scalar conservation laws satisfies the maximum principle, which means that it always stays within the invariant region
${\G}=[\min_{\bx}\bU(\bx,0),\max_{\bx}\bU(\bx,0)]$. For the Euler equations of gas dynamics, the corresponding invariant region
${\G}$ includes all admissible states with positive density and pressure. For notational convenience, we assume that the convex
invariant region ${\G}$ of equation \eqref{1} can be expressed as 
\begin{equation*}
	{\G}:=\big\{\,\bU\in\mathbb R^m: \varphi_i(\bU)>0, i\in\mathbb I~~\mbox{and}~~\varphi_i(\bU)\ge0, i\in\hat{\mathbb I}\,\big\}
\end{equation*}
with continuous functions $\{\varphi_i(\bU)\}$. It is highly desirable or even crucial to ensure that the numerical solutions of \eqref{1}
also stay within the invariant region ${\G}$. 

It is well-known that for scalar conservation laws, first-order monotone schemes are bound-preserving (BP) under a suitable CFL condition.
Some of these first-order schemes have also been shown to be BP for many hyperbolic systems. However, constructing high-order accurate BP\
schemes is a considerably more challenging task. A general framework for constructing high-order BP discontinuous Galerkin and
finite-volume schemes was proposed in \cite{ZhangShu2010_PP,zhang2010} for rectangular meshes, and later extended to triangular meshes in
\cite{ZhangXiaShu2012}. In this framework, high-order schemes were reformulated as convex combinations of first-order BP schemes, resulting
in the BP property of high-order schemes through the convexity of the invariant region. This framework has been extensively applied to
various hyperbolic systems; see, e.g.,
\cite{xing2010positivity,wang2012robust,zhang2013maximum,QinShu2016,ZHANG2017301,Wu2017,jiang2018invariant,du2019high,du2019third}. Another
approach for constructing high-order BP schemes is the BP flux limiting approach \cite{Xu2014,hu2013positivity}, which aims to design a
suitable convex combination of a first-order BP flux and a high-order flux to achieve both the BP property and a high order of accuracy
simultaneously. This approach has been successfully applied to various conservation laws, including scalar conservation laws \cite{Xu2014},
compressible Euler equations \cite{hu2013positivity,xiong2016parametrized}, and special relativistic hydrodynamics \cite{WuTang2015}. For
more developments on high-order BP schemes, we refer the reader to \cite{zhang2011b,XuZhang2017,guermond2017invariant,yan2023efficient}.

Another challenge in analyzing and constructing BP schemes lies in the complexity of invariant region ${\G}$, which may exhibit very
complicated explicit formula involving high nonlinearity, not to mention the cases where it can only be defined implicitly \cite{Wu2021GQL}.
Recently, motivated by a series of BP efforts \cite{WuTangM3AS,Wu2017a,WuShu2018,WuShu2019,WuShu2020NumMath} for magnetohydrodynamic (MHD)
equations, the geometric quasilinearization (GQL) framework was proposed in \cite{Wu2021GQL} to address BP problems involving highly
nonlinear or even implicit constraints. The GQL approach was also applied in \cite{wu2021minimum} to design BP schemes that preserve the
invariant region for relativistic hydrodynamics. In addition, the GQL framework was utilized in \cite{WuJiangShu2022} to explore BP central
discontinuous Galerkin schemes for the ideal MHD equations. 

Central-upwind (CU) schemes belong to the class of Riemann-problem-solver-free Godunov-type schemes. Compared with the staggered central
schemes \cite{AVM,JT,LPR99,LT,NessyahuTadmor1990}, CU schemes incorporate an upwinding information on local speeds of propagation to reduce
the amount of numerical dissipation present in staggered central schemes and to admit a particularly simple semi-discrete form. CU schemes
were introduced in \cite{Kurganov2001,KTcl,KTrp} and then, in \cite{KurganovLin2007}, where a more accurate projection step was proposed,
leading to a ``built-in'' anti-diffusion term. The numerical dissipation present in CU schemes was further reduced by modifying the CU
numerical fluxes using the local characteristic decomposition \cite{CCHKL} and by implementing a subcell approach in the projection step,
which leads to a modified, more sharper ``built-in'' anti-diffusion term \cite{KurXin}. The latter two CU schemes, however, cannot be viewed
as ``black-box'' solvers for general hyperbolic systems of conservation laws as they are constructed for each particular system. Therefore,
in this paper, we focus on the CU schemes from \cite{KurganovLin2007} and their two-dimensional (2-D) modification from
\cite{ChertockCuiKurganovOzcanTadmor2018}.

CU schemes have been applied to various hyperbolic conservation laws or balance laws and it was shown that they satisfy some of the BP
properties such as positivity preserving of the density for the Euler equations of gas dynamics and water depth for a variety of shallow 
water and related models. This paper presents a novel generic framework for analyzing and constructing BPCU schemes for general hyperbolic
systems of conservation laws. The main challenges in achieving this goal arise from the presence of the ``built-in'' anti-diffusion terms,
which makes it impossible to recast the CU schemes as convex combinations of first-order BP schemes, thereby invalidating some standard BP
analysis techniques. In order to overcome these challenges, we introduce a novel convex decomposition of the CU schemes and find a
sufficient condition of the desired BP property. Thanks to this key discovery, the goal of constructing BPCU schemes is simplified into four
more accessible tasks, each of which can be accomplished by suitable adjustments to the original CU schemes. Our findings in this paper
include the following:

\begin{itemize}[leftmargin = 1em]
	
\item We propose a novel convex decomposition of the CU schemes, which rewrites them as a convex combination of several intermediate solution states. This formulation overcomes the challenge posed by the complex formula of the CU numerical fluxes and facilitates the BP
property analysis;

\item We establish a sufficient condition for the BP property of the CU schemes, which only requires the BP property of the intermediate solution states. This finding decomposes the goal of constructing BPCU schemes into four more accessible tasks;

\item  We develop a BP framework for CU schemes that summarizes our findings. This framework includes a step-by-step procedure for modifying the original CU schemes into BPCU schemes, as well as a set of theoretical results that can be used to prove rigorously the BP property of the BPCU schemes. This BP framework is applicable to general hyperbolic systems of conservation laws;

\item We apply the proposed BP framework to construct provably BPCU schemes for the Euler equations of gas dynamics. The analysis of these BPCU schemes is novel and nontrivial, involving both technical estimates and the GQL approach;

\item We implement our BPCU schemes and demonstrate their robustness and effectiveness on several demanding numerical examples including high-speed jet problems, flow past a forward-facing step, and a shock diffraction problem. 
\end{itemize}

\smallskip
This paper is organized as follows. In \S\ref{sec2}, we discuss the one-dimensional (1-D) BPCU schemes. Specifically, we provide a review of
the 1-D CU schemes in \S\ref{sec21}; propose the BP framework for the 1-D CU schemes in \S\ref{sec22}; and utilize the BP framework to
construct the BPCU scheme for 1-D Euler equations of gas dynamics in \S\ref{sec23}. In \S\ref{sec3}, we extend the BP framework to the
2-D case. Several numerical tests are presented in \S\ref{sec4} to demonstrate the effectiveness and robustness of the proposed BPCU
schemes. Finally, we conclude the paper in \S\ref{sec5} by summarizing our contributions.

\section{One-Dimensional BPCU Schemes}\label{sec2}
\subsection{CU Schemes: a Brief Overview}\label{sec21}
In this section, we provide a brief overview of the 1-D CU scheme from \cite{KurganovLin2007} for
\begin{equation}
	\bU_t+\bF(\bU)_x=\bm0.
	\label{1f}
\end{equation}

Let $\big\{\Omega_j:=[x_{j-\frac12},x_{j+\frac12}]\big\}$ be a uniform partition of the 1-D spatial domain with the spatial step-size
$\dx=x_{j+\frac12}-x_{j-\frac12}$. We denote by $\,\xbar\bU^{\,n}_j$ the approximation of the cell average of $\bU(x,t)$ in the cell
$\Omega_j$ at the time level $t=t^n$. With the discussion on high-order temporal discretizations deferred to \Cref{rem24}, we first
consider the semi-discrete CU scheme coupled with forward Euler temporal discretization, namely,
\begin{equation}
	\xbar\bU^{\,n+1}_j=\,\xbar\bU^{\,n}_j-\lambda^n\big(\hat\bF_{j+\frac12}-\hat\bF_{j-\frac12}\big),\quad\lambda^n:=\frac{\dt^n}{\dx},~~
	\dt^n:=t^{n+1}-t^n
	\label{2}
\end{equation}
with the CU numerical fluxes
\begin{equation}
	\hat\bF_{j+\frac12}=\frac{\sigma^+_{j+\frac12}\bF\big(\bU^-_{j+\frac12}\big)-\sigma^-_{j+\frac12}\bF\big(\bU^+_{j+\frac12}\big)}
	{\sigma^+_{j+\frac12}-\sigma^-_{j+\frac12}}+\frac{\sigma^+_{j+\frac12}\sigma^-_{j+\frac12}}{\sigma^+_{j+\frac12}-\sigma^-_{j+\frac12}}
	\big(\bU^+_{j+\frac12}-\bU^-_{j+\frac12}-\bd_{j+\frac12}\big).
	\label{3}
\end{equation}
Notice that the numerical fluxes $\hat\bF_{j+\frac12}$ are computed at the time level $t=t^n$, but we omit the time-dependence of all of the
half-integer index quantities in \eqref{3} and below for the sake of brevity. In \eqref{3}, $\bU^-_{j+\frac12}$ and $\bU^+_{j+\frac12}$ are
the left- and right-sided values of the reconstructed solution at the cell interface $x=x_{j+\frac12}$. If a piecewise linear reconstruction
is used, which is the case when one is interested in designing a second-order scheme, these one-sided point values are given by
\begin{equation}
	\bU^-_{j+\frac12}=\,\xbar{\bU}_j^{\,n}+\frac{\dx}{2}(\bU_x)_j^n,\quad
	\bU^+_{j+\frac12}=\,\xbar{\bU}_{j+1}^{\,n}-\frac{\dx}{2}(\bU_x)_{j+1}^n.
	\label{4}
\end{equation}
In order to prevent large magnitude oscillations near the discontinuities in the numerical solution, the slopes of the piecewise linear
reconstruction $(\bU_x)_j^n$ are typically computed using a nonlinear limiter. For instance, one may use the generalized minmod limiter
(see, e.g., \cite{Sweby1984,LieNoelle2003,NessyahuTadmor1990}):
\begin{equation}
	(\bU_x)_j^n={\rm minmod}\left(\theta\,\frac{\xbar{\bU}_j^{\,n}-\,\xbar{\bU}_{j-1}^{\,n}}{\dx},
	\frac{\xbar{\bU}_{j+1}^{\,n}-\,\xbar{\bU}_{j-1}^{\,n}}{2\dx},\theta\,\frac{\xbar{\bU}_{j+1}^{\,n}-\,\xbar{\bU}_j^{\,n}}{\dx}\right),
	\label{5}
\end{equation}
where $\theta\in[1,2]$ is a parameter that regulates the amount of numerical dissipation: Larger values of $\theta$ generally lead to a less
dissipative resulting scheme. The minmod function is defined by
$$
{\rm minmod}(c_1,c_2,\ldots):=\begin{cases}\min\{c_1,c_2,\ldots\}&\mbox{ if } c_i>0~~\forall i,\\
	\max\{c_1,c_2,\ldots\}&\mbox{ if } c_i<0~~\forall i,\\0&~~\mbox{otherwise,}\end{cases}
$$
and is applied in \eqref{5} in a component-wise manner.

The quantities $\sigma^\pm_{j+\frac12}$ in \eqref{3} denote the one-sided local speeds of propagation, and they can be estimated by 
\begin{equation}
	\sigma^-_{j+\frac12}=\min\big\{\lambda_1\big(\bU^-_{j+\frac12}\big),\lambda_1\big(\bU^+_{j+\frac12}\big),0\big\},\quad
	\sigma^+_{j+\frac12}=\max\big\{\lambda_m\big(\bU^-_{j+\frac12}\big),\lambda_m\big(\bU^+_{j+\frac12}\big),0\big\},
	\label{6}
\end{equation}
where $\lambda_1(\bU)$ and $\lambda_m(\bU)$ are the smallest and largest eigenvalues of the Jacobian $\partial\bF/\partial\bU$. We note that
for some hyperbolic systems (for instance, for the MHD equations \cite{Wu2017a,WuShu2019}), one may need to slightly underestimate 
$\sigma^-_{j+\frac12}$ and overestimate $\sigma^+_{j+\frac12}$ to ensure the BP property.
\begin{rmk}\label{rem21}
	As we need to divide by $\sigma^+_{j+\frac12}-\sigma^-_{j+\frac12}$ in the CU numerical flux \eqref{3}, the computation of the one-sided
	local speeds of propagation \eqref{6} must be desingularized. For instance, for those $j$ at which
	$\sigma^+_{j+\frac12}-\sigma^-_{j+\frac12}<\varepsilon$, one may replace $\sigma^\pm_{j+\frac12}$ with
	$\sigma^\pm_{j+\frac12}=\pm\varepsilon$, where $\varepsilon$ is a very small positive number.
\end{rmk}

Finally, the quantity $\bd_{j+\frac12}$ in \eqref{3} denotes the ``built-in'' anti-diffusion and is given by
\begin{equation}
	\bd_{j+\frac12}:={\rm minmod}\,\big(\bU^+_{j+\frac12}-\bU^*_{j+\frac12},\bU^*_{j+\frac12}-\bU^-_{j+\frac12}\big),
	\label{7}
\end{equation}
where
\begin{equation}
	\begin{aligned}
		\bU^*_{j+\frac12}=&R_f\big(\sigma^+_{j+\frac12},\sigma^-_{j+\frac12},\bU^+_{j+\frac12},\bU^-_{j+\frac12}\big),\\
		R_f\big(\sigma^+,\sigma^-,\bU^+,\bU^-\big):=&\frac{\sigma^+\bU^+-\sigma^-\bU^{-}-\bF(\bU^+)+\bF(\bU^-)}{\sigma^+-\sigma^-}.
	\end{aligned}
	\label{8}
\end{equation}
\begin{rmk}
	It was shown in \cite{BKLP} that the CU numerical flux \eqref{3}--\eqref{8} for scalar conservation laws is monotone. According to
	\cite{zhang2010}, this implies that the CU scheme \eqref{2}--\eqref{8} for scalar conservation laws is provably BP under the CFL condition
	\eqref{11}. In fact, it is also the case if the high-order strong stability preserving (SSP) Runge-Kutta or multistep method
	(\!\!\cite{GottliebKetchesonShu2011,GST}) is used instead of the forward Euler method for time integration.
\end{rmk}

\subsection{BP Framework for CU Schemes}\label{sec22}
The CU scheme reviewed in \S\ref{sec21} is {\em not} BP when applied to general hyperbolic systems of conservation laws. In this section,
we propose a systematic framework for constructing BPCU schemes for general hyperbolic systems \eqref{1f}. 

Due to the presence of the anti-diffusion terms $\bm d_{j+\frac12}$ in the CU numerical fluxes \eqref{3}, the standard BP approach, which is
typically based on rewriting a high-order numerical flux into a convex combination of formally first-order BP fluxes, becomes invalid in the
system case. In order to overcome this difficulty, we will first establish a novel technical convex decomposition to analyze the BP property
of the CU scheme and identify the challenges in achieving the desired BP property. We will then propose a series of modifications of the CU
scheme resulting in provably BPCU schemes.

We begin with introducing the notations
\begin{equation}
	\bU^{*,\pm}_{j+\frac12}:=\bU^*_{j+\frac12}-\frac{\sigma^\pm_{j+\frac12}}{\sigma^+_{j+\frac12}-\sigma^-_{j+\frac12}}\bd_{j+\frac12},\quad
	\bU^*_j:=R_f\big(\sigma^+_{j+\frac12},\sigma^-_{j-\frac12},\bU^-_{j+\frac12},\bU^+_{j-\frac12}\big),
	\label{9}
\end{equation}
and establish the following theorem, which is crucial for the development of BPCU schemes.
\begin{theorem}\label{thm21}
	The CU scheme \eqref{2}--\eqref{8} admits the following convex decomposition:
	\begin{equation}
		\begin{aligned} 
			\xbar\bU^{\,n+1}_j&=\Big(\frac12-\lambda^n\big(\sigma^+_{j-\frac12}-\sigma^-_{j-\frac12}\big)\Big)\bU^+_{j-\frac12}+
			\lambda^n\sigma^+_{j-\frac12}\bU^{*,-}_{j-\frac12}\\
			&+\Big(\frac12-\lambda^n\big(\sigma^+_{j+\frac12}-\sigma^-_{j+\frac12}\big)\Big)
			\bU^-_{j+\frac12}-\lambda^n\sigma^-_{j+\frac12}\bU^{*,+}_{j+\frac12}
			+\lambda^n\big(\sigma^+_{j+\frac12}-\sigma^-_{j-\frac12}\big)\,\bU_j^*
		\end{aligned}
		\label{10}
	\end{equation}
	under the CFL condition
	\begin{equation}
		\lambda^n\sigma\le\frac12,\quad\sigma:=\max_j\big\{\sigma^+_{j+\frac12}-\sigma^-_{j+\frac12}\big\}.
		\label{11}
	\end{equation}
\end{theorem}
\begin{proof}
	We first rewrite the numerical flux \eqref{3} as follows:
	\begin{equation}
		\begin{aligned}
			\hat\bF_{j+\frac12}&=\bF\big(\bU^-_{j+\frac12}\big)-\sigma^-_{j+\frac12}\bU^-_{j+\frac12}\\
			&\hspace*{-0.2cm}+\sigma^-_{j+\frac12}\left(\frac{\sigma^+_{j+\frac12}\bU^+_{j+\frac12}-\sigma^-_{j+\frac12}\bU^-_{j+\frac12}-
				\bF\big(\bU^+_{j+\frac12}\big)+\bF\big(\bU^-_{j+\frac12}\big)}
			{\sigma^+_{j+\frac12}-\sigma^-_{j+\frac12}}-\frac{\sigma^+_{j+\frac12}}{\sigma^+_{j+\frac12}-\sigma^-_{j+\frac12}}\,\bd_{j+\frac12}\right)\\
			&=\bF\big(\bU^-_{j+\frac12}\big)-\sigma^-_{j+\frac12}\bU^-_{j+\frac12}+\sigma^-_{j+\frac12}
			\Big(\bU^*_{j+\frac12}-\frac{\sigma^+_{j+\frac12}}{\sigma^+_{j+\frac12}-\sigma^-_{j+\frac12}}\,\bd_{j+\frac12}\Big)\\
			&=\bF\big(\bU^-_{j+\frac12}\big)-\sigma^-_{j+\frac12}\bU^-_{j+\frac12}+\sigma^-_{j+\frac12}\bU_{j+\frac12}^{*,+}.
		\end{aligned}
		\label{12}
	\end{equation}
	Similarly, we obtain
	\begin{equation}
		\begin{aligned}
			\hat\bF_{j-\frac12}&=\bF\big(\bU^+_{j-\frac12}\big)-\sigma^+_{j-\frac12}\bU^+_{j-\frac12}\\
			&\hspace*{-0.2cm}+\sigma^+_{j-\frac12}\left(\frac{\sigma^+_{j-\frac12}\bU^+_{j-\frac12}-\sigma^-_{j-\frac12}\bU^-_{j-\frac12}-
				\bF\big(\bU^+_{j-\frac12}\big)+\bF\big(\bU^-_{j-\frac12}\big)}{\sigma^+_{j-\frac12}-\sigma^-_{j-\frac12}}-\frac{\sigma^-_{j-\frac12}}
			{\sigma^+_{j-\frac12}-\sigma^-_{j-\frac12}}\,\bd_{j-\frac12}\right)\\
			&=\bF\big(\bU^+_{j-\frac12}\big)-\sigma^+_{j-\frac12}\bU^+_{j-\frac12}+\sigma^+_{j-\frac12}\Big(\bU^*_{j-\frac12}-
			\frac{\sigma^-_{j-\frac12}}{\sigma^+_{j-\frac12}-\sigma^-_{j-\frac12}}\,\bd_{j-\frac12}\Big)\\
			&=\bF\big(\bU^+_{j-\frac12}\big)-\sigma^+_{j-\frac12}\bU^+_{j-\frac12}+\sigma^+_{j-\frac12}\bU_{j-\frac12}^{*,-}.
			\label{13}
		\end{aligned}
	\end{equation}
	Substituting the reformulated numerical fluxes \eqref{12} and \eqref{13} into \eqref{2} and using \eqref{4} yields
	$$
	\begin{aligned}
		&\xbar\bU_j^{\,n+1}=\frac12\big(\bU^+_{j-\frac12}+\bU^-_{j+\frac12}\big)-\lambda^n\big(\bF\big(\bU^-_{j+\frac12}\big)-
		\sigma^-_{j+\frac12}\bU^-_{j+\frac12}+\sigma^-_{j+\frac12}\bU_{j+\frac12}^{*,+}\big)\\
		&\hspace*{0.85cm}+\lambda^n\big(\bF\big(\bU^+_{j-\frac12}\big)-
		\sigma^+_{j-\frac12}\bU^+_{j-\frac12}+\sigma^+_{j-\frac12}\bU_{j-\frac12}^{*,-}\big)\\
		&=\Big(\frac12-\lambda^n\big(\sigma^+_{j-\frac12}-\sigma^-_{j-\frac12}\big)\Big)\bU^+_{j-\frac12}+
		\lambda^n\sigma^+_{j-\frac12}\bU^{*,-}_{j-\frac12}+\Big(\frac12-\lambda^n\big(\sigma^+_{j+\frac12}-\sigma^-_{j+\frac12}\big)\Big)
		\bU^-_{j+\frac12}\\
		&-\lambda^n\sigma^-_{j+\frac12}\bU^{*,+}_{j+\frac12}
		+\lambda^n\big(\sigma^+_{j+\frac12}-\sigma^-_{j-\frac12}\big)\,\frac{\sigma^+_{j+\frac12}\bU^-_{j+\frac12}-
			\sigma^-_{j-\frac12}\bU^+_{j-\frac12}-\bF\big(\bU^-_{j+\frac12}\big)+\bF\big(\bU^+_{j-\frac12}\big)}
		{\sigma^+_{j+\frac12}-\sigma^-_{j-\frac12}}\\
		&=\Big(\frac12-\lambda^n\big(\sigma^+_{j-\frac12}-\sigma^-_{j-\frac12}\big)\Big)\bU^+_{j-\frac12}+
		\lambda^n\sigma^+_{j-\frac12}\bU^{*,-}_{j-\frac12}+\Big(\frac12-\lambda^n\big(\sigma^+_{j+\frac12}-\sigma^-_{j+\frac12}\big)\Big)
		\bU^-_{j+\frac12}\\
		&-\lambda^n\sigma^-_{j+\frac12}\bU^{*,+}_{j+\frac12}+\lambda^n\big(\sigma^+_{j+\frac12}-\sigma^-_{j-\frac12}\big)\,\bU_j^*,
	\end{aligned}
	$$
	so that $\xbar\bU_j^{\,n+1}$ has been reformulated as a convex combination of $\bU^\mp_{j\pm\frac12}$, $\bU^{*,\pm}_{j\pm\frac12}$, and
	$\bU^*_j$ under the CFL condition \eqref{11}. The proof of the theorem is thus completed.
\end{proof}

Thanks to the convexity of ${\G}$, \Cref{thm21} immediately leads to the following sufficient condition for obtaining BPCU schemes.
\begin{corollary}\label{cor21}
	If $\bU^\pm_{j+\frac12}\in{\G}$, $\bU^{*,\pm}_{j+\frac12}\in{\G}$, $\bU^*_j\in{\G}$, for all $j$, then
	$\xbar\bU_j^{\,n+1}\in{\G}$ for all $j$ and the CU scheme \eqref{2}--\eqref{3} is BP under the CFL condition \eqref{11}.
\end{corollary}

Inspired by \Cref{cor21}, we simplify the challenging goal of designing BPCU schemes into four more accessible tasks, that is, modifying the
CU schemes in such a way that the following four essential conditions are satisfied:
\begin{itemize}[leftmargin=42mm]
	\setlength\itemsep{0mm}
	\item[\bf 1-D BP Condition \#1:] If $\,\xbar\bU_j^{\,n}\in{\G}$ for all $j$, then $\bU^\pm_{j+\frac12}\in{\G}$ for all $j$;
	\item[\bf 1-D BP Condition \#2:] If $\bU^\pm_{j+\frac12}\in{\G}$ for all $j$, then $\bU^*_{j+\frac12}\in{\G}$ for all $j$;
	\item[\bf 1-D BP Condition \#3:] If $\bU^\pm_{j+\frac12}\in{\G}$ for all $j$, then $\bU^*_j\in{\G}$ for all $j$;
	\item[\bf 1-D BP Condition \#4:] If $\bU^\pm_{j+\frac12}\in{\G}$, $\bU^*_{j+\frac12}\in{\G}$ for all $j$, then
	$\bU^{*,\pm}_{j+\frac12}\in{\G}$ for all $j$.
\end{itemize}
In \S\ref{sec221}--\ref{sec224}, we will propose suitable modifications to simultaneously enforce these conditions.
\begin{rmk}
	Although $\bU^*_{j+\frac12}$ is not directly required in \eqref{10}, BP Condition \#2 is included as an indispensable step towards enforcing
	BP Condition \#4. Note that $\bU^*_{j+\frac12}$ is a convex combination of $\bU^{*,-}_{j+\frac12}$ and ${\bU}^{*,+}_{j+\frac12}$ since
	$
	\bU^*_{j+\frac12}=\frac{\sigma^+_{j+\frac12}\bU^{*,-}_{j+\frac12}-\sigma^-_{j+\frac12}\bU^{*,+}_{j+\frac12}}
	{\sigma^+_{j+\frac12}-\sigma^-_{j+\frac12}}.
	$
	Thus, due to the convexity of ${\G}$, BP Condition \#2 is a necessary condition for enforcing BP Condition \#4.
\end{rmk}

\subsubsection{BP Condition \#1: \texorpdfstring{$\bU^\pm_{j+\frac12}\in{\G}$}{$u^\pm_{j+\frac12}\in{\G}$}}\label{sec221}
We first define
$$
{\G}_\varepsilon:=\big\{\,\bU\in\mathbb R^m: \varphi_i(\bU)\ge\varepsilon_i, i\in\mathbb I~~\mbox{and}~~\varphi_i(\bU)\ge0,
i\in\hat{\mathbb I}\,\big\}\subset{\G},
$$
where the parameters $\varepsilon_i=\min\{10^{-13},\varphi_i(\,\xbar\bU_j^{\,n})\}$ for $i\in\mathbb I$ are introduced to avoid the effect
of round-off errors. It should be noted that ${\G}_\varepsilon$ may not be convex, but very close to the convex set ${\G}$. If
either $\bU^+_{j-\frac12}\not\in{\G}_\varepsilon$ or $\bU^-_{j+\frac12}\not\in{\G}_\varepsilon$, we use a local scaling BP limiter
\cite{ZhangShu2010_PP} to replace these point values with
\begin{equation*}
	\tilde\bU^+_{j-\frac12}=\,\xbar\bU_j^{\,n}+\delta_j\,\big(\bU^+_{j-\frac12}-\,\xbar\bU_j^{\,n}\big)\in{\G}_\varepsilon\quad\mbox{and}
	\quad
	\tilde\bU^-_{j+\frac12}=\,\xbar\bU_j^{\,n}+\delta_j\,\big(\bU^-_{j+\frac12}-\,\xbar\bU_j^{\,n}\big)\in{\G}_\varepsilon,
\end{equation*}
where $\delta_j=\Theta(\,\xbar\bU^{\,n}_j,\bU^+_{j-\frac12},\bU^-_{j+\frac12})$. The function $\Theta$ can be defined, for instance, as  
\begin{equation}
	\Theta(\,\xbar\bU,\bU^+,\bU^-)=\min\left\{\,\frac{\|\bm s_\varepsilon^+-\xbar\bU\|_2}{\|\bU^+-\xbar\bU\|_2},
	\frac{\|\bm s_\varepsilon^--\xbar\bU\|_2}{\|\bU^--\xbar\bU\|_2}\,\right\},
	\label{15}
\end{equation}
where the points $\bm s_\varepsilon^\pm$ are obtained as follows. Denote by $\psi^-$ the line segment connecting $\xbar\bU$ and $\bU^-$ and
by $\psi^+$ the line segment connecting $\xbar\bU$ and $\bU^+$. Then:

\noindent
$\bullet$ $\bm s_\varepsilon^-$ is the closest to $\xbar\bU$ point among $\bU^-$ and the intersection points (if any) of $\psi^-$ and
$\partial{\G}_\varepsilon$;

\noindent
$\bullet$ $\bm s_\varepsilon^+$ is the closest to $\xbar\bU$ point among $\bU^+$ and the intersection points (if any) of $\psi^+$ and
$\partial{\G}_\varepsilon$;

\noindent
as illustrated in \Cref{fig1}.
\begin{figure}[ht!]
	\centerline{\includegraphics[width=0.38\textwidth]{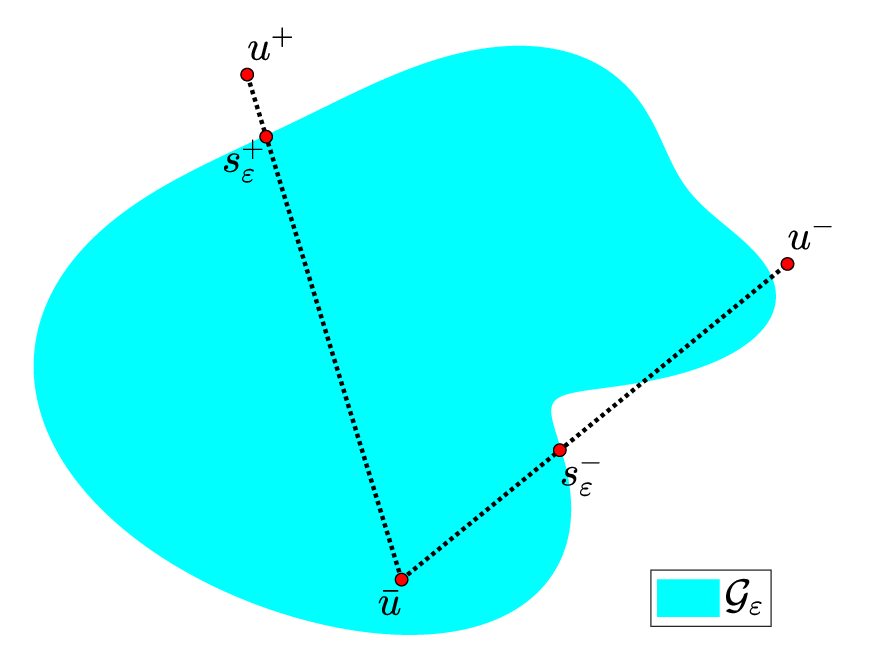}}
	\caption{\sf A possible configuration of $\,\xbar\bU$, $\bU^\pm$, and $\bm s^\pm_\varepsilon$ (not to scale).\label{fig1}}
\end{figure}

\begin{rmk}
	In many cases, the calculation of $\bm s^\pm_\varepsilon$ is cumbersome and may require a certain root-finding procedure. If the functions
	$\varphi_i(\bU)$ are all concave with respect to $\bU$, then the function $\Theta$ can be alternatively defined as
	\begin{equation}
		\begin{aligned}
			&\Theta(\,\xbar\bU,\bU^+,\bU^-)=\min_{i\in\mathbb I\cup\hat{\mathbb I}}\delta_i,\quad
			\tilde\varepsilon_i=\begin{cases}\varepsilon_i&\mbox{if}~i\in\mathbb I,\\0&\mbox{if}~i\in\hat{\mathbb I},\end{cases}\\
			&\delta_i=\begin{dcases}\frac{\varphi_i(\,\xbar\bU)-\tilde\varepsilon_i}{\varphi_i(\,\xbar\bU)-\min\{\varphi_i(\bU^+),\varphi_i(\bU^-)\}}&
				\mbox{if}~\min\{\varphi_i(\bU^+),\varphi_i(\bU^-)\}<\tilde\varepsilon_i,\\1&\mbox{otherwise}.\end{dcases}
		\end{aligned}
		\label{16}
	\end{equation}
\end{rmk}

\subsubsection{BP Condition \#2: \texorpdfstring{$\bU^*_{j+\frac12}\in{\G}$}{$u^*_{j+\frac12}\in{\G}$}}\label{sec222}
We begin by proving the following two auxiliary lemmas.
\begin{lemma}\label{lem22}
	If $\,\bU^-,\bU^+\in{\G}$, $\,\hat\sigma^-\le\sigma^-$, and $\,\hat\sigma^+\ge\sigma^+$, then
	$$R_f(\sigma^+,\sigma^-,\bU^+,\bU^-)\in{\G}~\Longrightarrow~R_f(\hat\sigma^+,\hat\sigma^-,\bU^+,\bU^-)\in{\G}.$$
\end{lemma}
\begin{proof}
	We first introduce the following notation: $\varepsilon^+:=\hat\sigma^+-\sigma^+\ge0$ and $\varepsilon^-:=\sigma^--\hat\sigma^-\ge0$, and
	then use the definition of $R_f$ in \eqref{8} to obtain
	$$
	\begin{aligned}
		R_f(\hat\sigma^+,\hat\sigma^-,\bU^+,\bU^-)&=\frac{\hat\sigma^+\bU^+-\hat\sigma^-\bU^--\bF(\bU^+)+\bF(\bU^-)}{\hat\sigma^+-\hat\sigma^-}\\
		&=\frac{(\sigma^+-\sigma^-)\,R_f(\sigma^+,\sigma^-,\bU^+,\bU^-)+\varepsilon^+\bU^++\varepsilon^-\bU^-}
		{\sigma^+-\sigma^-+\varepsilon^++\varepsilon^-}.
	\end{aligned}
	$$
	This means that $R_f(\hat\sigma^+,\hat\sigma^-,\bU^+,\bU^-)$ is a convex combination of $u^-$, $u^+$, and
	$R_f(\sigma^+,\sigma^-,\bU^+,\bU^-)$, which implies  that $R_f(\hat\sigma^+,\hat\sigma^-,\bU^+,\bU^-)\in{\G}$ due to the convexity of
	${\G}$. The proof is thus completed.	
\end{proof}
\begin{lemma}\label{lem23}
	For the system \eqref{1f} with a convex invariant region ${\G}$, there exist functions $\sigma^-(\bU^-,\bU^+)$ $ \le 0$ and
	$\sigma^+(\bU^-,\bU^+)\ge0$ satisfying the following condition:
	\begin{equation}
		R_f(\sigma^+(\bU^-,\bU^+),\sigma^-(\bU^-,\bU^+),\bU^+,\bU^-)\in{\G}~~\forall~\bU^-,\bU^+\in{\G}.
		\label{17}
	\end{equation}
\end{lemma}
\begin{proof}
	Consider the Riemann problem for the hyperbolic system \eqref{1f} with $\bU^-$ and $\bU^+$ being the left and right states at $t=0$. Denote
	the minimum and maximum speeds of the Riemann solution by $\sigma_{\min}(\bU^-,\bU^+)$ and $\sigma_{\max}(\bU^-,\bU^+)$. Due to the finite
	speed of propagation, the Riemann fan is enclosed by the interval
	$[T\sigma_{\min}(\bU^-,\bU^+),T\sigma_{\max}(\bU^-,\bU^+)]$ at time $t=T$. If we choose
	$\sigma^+(\bU^-,\bU^+):=\max\{\sigma_{\max}(\bU^-,\bU^+),0\}$ and $\sigma^-(\bU^-,\bU^+):=\min\{\sigma_{\min}(\bU^-,\bU^+),0\}$, then
	$R_f(\sigma^+(\bU^-,\bU^+),\sigma^-(\bU^-,\bU^+),\bU^+,\bU^-)$ is exactly the average of the solution of the Riemann problem at $t=T$ over
	the interval $[T\sigma^-,T\sigma^+]$. Since the solution of the Riemann problem $u(x,T)$ is contained in the invariant region ${\G}$,
	its average over $[T\sigma^-,T\sigma^+]$ is also in ${\G}$ due to the convexity of ${\G}$. Hence, the condition \eqref{17} is
	satisfied.
\end{proof}

Based on the results established in \Cref{lem22,lem23}, the BP Condition \#2 can be enforced by a proper choice of the one-sided local
speeds of propagation $\sigma^\pm_{j+\frac12}$ as stated in the following theorem.
\begin{theorem}\label{thm24}
	Given $\bU^\pm_{j+\frac12}\in{\G}$ for all $j$ and given functions $\sigma^+(\cdot,\cdot)$ and $\sigma^-(\cdot,\cdot)$ satisfying
	\eqref{17}, the BP Condition \#2 is satisfied if
	$$
	\sigma^+_{j+\frac12}\ge\sigma^+(\bU^-_{j+\frac12},\bU^+_{j+\frac12})\quad\mbox{and}\quad
	\sigma^-_{j+\frac12}\le\sigma^-(\bU^-_{j+\frac12},\bU^+_{j+\frac12})\quad\forall j.
	$$
\end{theorem}
\begin{proof}
	The condition \eqref{17} implies
	$R_f(\sigma^+(\bU^-_{j+\frac12},\bU^+_{j+\frac12}),\sigma^-(\bU^-_{j+\frac12},\bU^+_{j+\frac12}),\bU^+_{j+\frac12},\bU^-_{j+\frac12})\in
	{\G}$, which by \Cref{lem22} yields
	$\bU^*_{j+\frac12}=R_f(\sigma^+_{j+\frac12},\sigma^-_{j+\frac12},\bU^+_{j+\frac12},\bU^{-}_{j+\frac12})\in{\G}$. Hence, the BP
	Condition \#2 is satisfied.
\end{proof}

Based on the above analysis, the task of enforcing BP Condition \#2 reduces to finding suitable functions $\sigma^-(\bU^-,\bU^+)$ and
$\sigma^+(\bU^-,\bU^+)$ satisfying \eqref{17}. This may be not easy as it is often very difficult to exactly evaluate the minimum and
maximum speeds of the Riemann solution, $\sigma_{\min}(\bU^-,\bU^+)$ and $\sigma_{\max}(\bU^-,\bU^+)$. Alternatively, one may estimate a
lower bound $\hat\sigma^-(\bU^-,\bU^+)\le\min\{\sigma_{\min}(\bU^-,\bU^+),0\}$ and an upper bound
$\hat\sigma^+(\bU^-,\bU^+)\ge\max\{\sigma_{\max}(\bU^-,\bU^+),0\}$. Based on \Cref{lem22}, such functions $\hat\sigma^-(\bU^-,\bU^+)$ and
$\hat\sigma^+(\bU^-,\bU^+)$ also satisfy \eqref{17}. 

\subsubsection{BP Condition \#3: \texorpdfstring{$\bU^*_j\in{\G}$}{$u^*_j\in{\G}$}}\label{sec223}
\begin{theorem}\label{thm25}
	Given $\bU^\pm_{j+\frac12}\in{\G}$ for all $j$ and given functions $\sigma^+(\cdot,\cdot)$ and $\sigma^-(\cdot,\cdot)$ satisfying
	\eqref{17}, the BP Condition \#3 is satisfied if
	$$
	\sigma_{j+\frac12}^+\ge\sigma^+(\bU^+_{j-\frac12},\bU^-_{j+\frac12})\quad\mbox{and}\quad
	\sigma_{j-\frac12}^-\le\sigma^-(\bU^+_{j-\frac12},\bU^-_{j+\frac12})\quad\forall j.
	$$
\end{theorem}
\begin{proof}
	The condition \eqref{17} implies
	$R_f(\sigma^+(\bU^+_{j-\frac12},\bU^-_{j+\frac12}),\sigma^-(\bU^+_{j-\frac12},\bU^-_{j+\frac12}),\bU^-_{j+\frac12},\bU^+_{j-\frac12})\in
	{\G}$, which by
	\Cref{lem22} yields 
	$$
	\bU^*_j =R_f(\sigma^+_{j+\frac12},\sigma^-_{j-\frac12},\bU^-_{j+\frac12},\bU^+_{j-\frac12})\in{\G}.
	$$ 
	Hence, the BP
	Condition \#3 is satisfied.
\end{proof}

According to \Cref{thm24,thm25}, both the BP Conditions \#2 and \#3 can be enforced under certain conditions on the one-sided local speeds 
of propagation $\{\sigma_{j+\frac12}^\pm\}$ as stated in the following theorem.
\begin{theorem}\label{thm26}
	Given $\bU^\pm_{j+\frac12}\in{\G}$ for all $j$ and given functions $\sigma^+(\cdot,\cdot)$ and $\sigma^-(\cdot,\cdot)$ satisfying
	\eqref{17}, both the BP Conditions \#2 and \#3 are satisfied if
	$$
	\begin{aligned}
		&\sigma^+_{j+\frac12}=\max\Big\{\sigma^+(\bU^-_{j+\frac12},\bU^+_{j+\frac12}),\sigma^+(\bU^+_{j-\frac12},\bU^-_{j+\frac12})\Big\},\\
		&\sigma^-_{j+\frac12}=\min\Big\{\sigma^-(\bU^-_{j+\frac12},\bU^+_{j+\frac12}),\sigma^-(\bU^+_{j+\frac12},\bU^-_{j+\frac32})\Big\},
	\end{aligned}\quad\forall j.
	$$
\end{theorem}
\Cref{thm26} is a direct consequence of \Cref{thm24,thm25}, and its proof is thus omitted.

\subsubsection{BP Condition \#4: \texorpdfstring{$\bU^{*,\pm}_{j+\frac12}\in{\G}$}{$u^{*,\pm}_{j+\frac12}\in{\G}$}}\label{sec224}
Let us consider the following linear function:
\begin{equation}
	\bm\psi(\xi)=\bU^*_{j+\frac12}+\frac{\bd_{j+\frac12}}{\sigma^+_{j+\frac12}-\sigma^-_{j+\frac12}}\,\xi, 
	\label{18}
\end{equation}
on the interval $\xi\in\big[-\sigma^+_{j+\frac12},-\sigma^-_{j+\frac12}\big]$. Notice that it follows from \eqref{9} that
$\bm\psi(-\sigma^\pm_{j+\frac12})=\bU^{*,\pm}_{j+\frac12}$ and thus, in order to enforce the BP Condition \#4 one may use a limited
modification of \eqref{18},
\begin{equation*}
	\widetilde{\bm\psi}(\xi)=\bU^*_{j+\frac12}+\Theta\big(\bU^*_{j+\frac12},\bU^{*,-}_{j+\frac12},\bU^{*,+}_{j+\frac12}\big)
	\frac{\bd_{j+\frac12}}{\sigma^+_{j+\frac12}-\sigma^-_{j+\frac12}}\,\xi, 
\end{equation*}
with one of the functions $\Theta$ introduced in \S\ref{sec224} in either \eqref{15} or \eqref{16}. It is easy to verify that
$\tilde\bU^{*,\pm}_{j+\frac12}:=\widetilde{\bm\psi}(-\sigma^\pm_{j+\frac12})\in{\G}$, so one has to replace the values
$\bU^{*,\pm}_{j+\frac12}$ with their limited versions $\tilde\bU^{*,\pm}_{j+\frac12}$.

\subsubsection{Constructing BPCU Schemes}\label{sec225}
The BPCU schemes for \eqref{1f} can be constructed via the following steps.

\noindent
{\bf Step 1.} Check whether the original CU scheme satisfies the BP Condition \#1. If not, modify the values $\bU_{j+\frac12}^\pm$ as
described in \S\ref{sec221}.

\noindent
{\bf Step 2.} Check whether the modified (after Step 1) CU scheme satisfies the BP Conditions \#2 and \#3. If not, modify the one-sided
local speeds of propagation $\sigma_{j+\frac12}^\pm$ as described in \S\ref{sec222} and \S\ref{sec223}.

\noindent
{\bf Step 3.} Check whether the modified (after Steps 1 and 2) CU scheme satisfies the BP Condition \#4. If not, modify the ``built-in''
anti-diffusion term $\bd_{j+\frac12}$ as described in \S\ref{sec224}.

\smallskip
We summarize the procedure of constructing BPCU schemes in \Cref{fig2}. 
Following these steps, one can construct BPCU schemes for a variety of hyperbolic systems of conservation laws. In \S\ref{sec23}, we
will construct provably BPCU schemes for the Euler equations of gas dynamics.
\begin{figure}[ht!]
	\centering
	\begin{tikzpicture}[scale = 0.8]
		\draw [-latex,thick] (3.1,-4) -- (3.1,-5) -- (6.03,-5);
		\draw [-latex,thick] (3.0,-4) -- (3.0,-6.1) -- (7.50,-6.1);
		\draw [-latex,thick] (2.9,-4) -- (2.9,-7.1) -- (10.5,-7.1);
		\draw [thick] (3.1,-4) -- (6.55,-4);
		\draw [thick] (6.65,-4) -- (11.9,-4) -- (11.9,-3);
		\draw [thick] (12.1,-7) -- (12.1,-3);
		\draw [thick] (0   ,-0.9) -- (13.0,-0.9);
		\draw [thick] (12  ,-2) -- (12  ,-1.1) -- (13.0,-1.1);
		\node [process, fill=blue!20] at (0   ,-1) (un_j) {$\xbar\bU_j^{\,n}$};
		\node [process, fill=blue!20] at (14.5,-1) (unp1_j) {$\xbar\bU_j^{\,n+1}$};
		\node [process, fill=blue!20] at (6.6 ,-2) (dt) {$\dt$};
		\node [process, fill=blue!20] at (12  ,-2) (f_jph) {$\hat\bF_{j+\frac12}$};
		\node [process, fill=blue!20] at (3   ,-4) (u_jph) {$\bU_{j+\frac12}^\pm$};
		\node [process, fill=blue!20] at (6.6 ,-5) (sigma_jph) {$\sigma_{j+\frac12}^\pm$};
		\node [process, fill=blue!20] at (9.0 ,-6) (us_jph) {$\bU_{j+\frac12}^*$};
		\node [process, fill=blue!20] at (12  ,-7) (d_jph) {$\bd_{j+\frac12}$};
		\path [connector] (un_j) |- (u_jph);
		\path [connector] (sigma_jph) -- (dt);
		\path [connector] (dt) |- (unp1_j);
		\path [connector] (sigma_jph) |- (us_jph);
		\path [connector] (sigma_jph) -| (f_jph);
		\path [connector] (us_jph) |- (d_jph);
		\node[draw,fill=yellow!40,rounded corners] at (1.15, -4) (L1) {$\bU_{j+\frac12}^\pm$};
		\node[draw,fill=yellow!40,rounded corners] at (4.70, -5) (L2) {$\sigma_{j+\frac12}^\pm$};
		\node[draw,fill=yellow!40,rounded corners] at (6.6, -3.3) (L3) {CFL condition};
		\node[draw,fill=yellow!40,rounded corners] at (10.25, -7) (L6) {$\bd_{j+\frac12}$};
		\node[draw,fill=gray!10,rounded corners] at (7.40, -6) (L4) {\eqref{9}};
		\node[draw,fill=gray!10,rounded corners] at (12, -3.25) (L5) {\eqref{3}};
		\node[draw,fill=gray!10,rounded corners] at (12.85, -1) (L7) {\eqref{2}};
	\end{tikzpicture}
	\caption{\sf Flow chart of BPCU schemes from $\,\xbar\bU^{\,n}_j$ to $\,\xbar\bU^{\,n+1}_j$. To guarantee the BP property, the four formulae
		(blocks in yellow) may require careful redesign/verification and may differ from the counterparts of the original CU schemes.\label{fig2}}
\end{figure}

\begin{rmk}\label{rem24}
	The BP analysis presented in \S\ref{sec2} is based on the first-order forward Euler time discretization. The above results can be directly
	extended to high-order SSP Runge-Kutta or multistep method (see, e.g., \cite{GottliebKetchesonShu2011,GST}), which can be expressed as a
	convex combination of forward Euler steps. The BP property of the resulting BPCU schemes remains valid thanks to the convexity of
	${\G}$.
\end{rmk}

\subsection{BPCU schemes for 1-D Euler Equations of Gas Dynamics}\label{sec23}
The 1-D Euler equations of gas dynamics for ideal gases read as \eqref{1f} with
\begin{equation}
	\bU=\big(\rho,\rho v,E\big)^\top,\quad\bF(\bU)=\big(\rho v,\rho v^2+p,(E+p)v\big)^\top,\quad E=\frac12\rho v^2+\frac{p}{\gamma-1}.
	\label{20}
\end{equation}
Here, $\rho$ is the density, $v$ is the velocity, $E$ is the total energy, $p$ is the pressure, and the specific heat ratio $\gamma>1$ is a
constant. The speed of sound is given by $c=\sqrt{\gamma p/\rho}$, and the three eigenvalues of the Jacobian $\partial\bF/\partial\bU$ are 
$\lambda_1=v-c$, $\lambda_2=v$, and $\lambda_3=v+c$. A convex invariant region for the 1-D Euler equations of gas dynamics is
\begin{equation}
	{\G}=\Big\{\,\bU\in\mathbb R^3: \rho>0,~p(\bU)=(\gamma-1)\Big(E-\frac12\rho v^2\Big)>0\,\Big\}.
	\label{22}
\end{equation}

For the 1-D Euler equations of gas dynamics, the original CU schemes reviewed in \S\ref{sec21} are generally not BP. In this subsection, we
will use the BP framework proposed in \S\ref{sec22} to modify the original CU schemes such that the BP Conditions \#1--4 are satisfied,
resulting in provably BPCU schemes.

\subsubsection{Enforcing the BP Condition \#1}\label{sec231}
We first emphasize that for the Euler equations of gas dynamics, the minmod reconstruction \eqref{4}--\eqref{5} does not, in general,
satisfy the BP Condition \#1. More specifically, while the positivity of density is guaranteed, the reconsructed point values of pressure
may be negative. In order to address this issue, we adopt the local scaling BP limiter \cite{ZhangShu2010_PP,wang2012robust} to enforce the
positivity of pressure: we replace the point values $\bU^-_{j+\frac12}$ and $\bU^+_{j-\frac12}$ defined in \eqref{4} with
\begin{equation}
	\begin{aligned}
		&\tilde\bU^-_{j+\frac12}=\,\xbar\bU_j^{\,n}+\frac{\dx}{2}\delta_j\,(\bU_x)_j,
		\tilde\bU^+_{j-\frac12}=\,\xbar\bU_j^{\,n}-\frac{\dx}{2}\delta_j\,(\bU_x)_j,~
		\delta_j:=\min\bigg\{\frac{p(\,\xbar\bU_j^{\,n})-\hat\varepsilon_j}{p(\,\xbar{\bU}_j^{\,n})-p_j^{\min}},1\bigg\},\\
		&\hat\varepsilon_j:=\min\big\{10^{-13},p(\,\xbar\bU_j^{\,n})\big\},\quad p_j^{\min}:=\min\big\{p\big(\,\bU^-_{j+\frac12}\big),
		p\big(\,\bU^+_{j-\frac12}\big)\big\}.
	\end{aligned}
	\label{23}
\end{equation}

\begin{lemma}\label{lem27}
	For the 1-D Euler equations of gas dynamics \eqref{1f}, \eqref{20} with the invariant region \eqref{22}, the new limited values
	$\tilde\bU^\pm_{j+\frac12}$ defined in \eqref{23} satisfy the BP Condition \#1 provided $\,\xbar\bU^{\,n}_j\in{\G}$ for all $j$.
\end{lemma}
\begin{proof}
	First, we note that since $\,\xbar\rho^{\,n}_j>0$ for all $j$, the use of the minmod limiter ensures the positivity of
	$\rho^\pm_{j+\frac12}$ for all $j$, and thus $\tilde\rho^\pm_{j+\frac12}>0$ as $0\le\delta_j\le1$ for all $j$. Hence, we only need to prove
	that $p\big(\tilde\bU^\mp_{j\pm\frac12}\big)>0$ for all $j$.
	
	We consider two possible cases. If $p_j^{\min}\ge\hat\varepsilon_j$, \eqref{23} yields $\delta_j=1$ and thus
	$p\big(\tilde\bU^\mp_{j\pm\frac12}\big)=p\big(\bU^\mp_{j\pm\frac12}\big)\ge p_j^{\min}\ge\hat\varepsilon_j>0$. Otherwise, we apply Jensen's
	inequality to the pressure function $p(\bU)$, which is concave when $\rho>0$, and obtain
	$$
	\begin{aligned}
		p\big(\tilde\bU^\mp_{j\pm\frac12}\big)&=p\big(\,\xbar\bU_j^{\,n}+\delta_j\big(\bU^\mp_{j\pm\frac12}-\,\xbar\bU_j^{\,n}\big)\big)=
		p\big((1-\delta_j)\,\xbar\bU_j^{\,n}+\delta_j\,\bU^\mp_{j\pm\frac12}\big)\\
		&\ge(1-\delta_j)p(\,\xbar\bU_j^{\,n})+\delta_jp\big(\bU^\mp_{j\pm\frac12}\big)\\
		&\stackrel{\eqref{23}}{=}\bigg(1-\frac{p(\,\xbar\bU_j^{\,n})-\hat\varepsilon_j}{p(\,\xbar{\bU}_j^{\,n})-p_j^{\min}}\bigg)
		p(\,\xbar\bU_j^{\,n})+
		\frac{p(\,\xbar\bU_j^{\,n})-\hat\varepsilon_j}{p(\,\xbar{\bU}_j^{\,n})-p_j^{\min}}\,p\big(\bU^\mp_{j\pm\frac12}\big)\\
		&\ge\frac{\hat\varepsilon_j-p_j^{\min}}{p(\,\xbar{\bU}_j^{\,n})-p_j^{\min}}\,p(\,\xbar\bU_j^{\,n})+
		\frac{p(\,\xbar\bU_j^{\,n})-\hat\varepsilon_j}{p(\,\xbar{\bU}_j^{\,n})-p_j^{\min}}\,p_j^{\min}=\hat\varepsilon_j>0.
	\end{aligned}
	$$
	The proof is completed.
\end{proof}

\subsubsection{The BP Conditions \#2 and \#3}\label{sec232}
Equipped with the reconstructed point values
$\bU^\pm_{j+\frac12}=\big(\rho^\pm_{j+\frac12},(\rho v)^\pm_{j+\frac12},E^\pm_{j+\frac12}\big)^\top$, we compute the one-sided local speeds
of propagation using equation \eqref{6}, which for the Euler equations of gas dynamics reads as
\begin{equation}
	\sigma^-_{j+\frac12}=\min\big\{v^-_{j+\frac12}-c^-_{j+\frac12},v^+_{j+\frac12}-c^+_{j+\frac12},0\big\},~~
	\sigma^+_{j+\frac12}=\max\big\{v^-_{j+\frac12}+c^-_{j+\frac12},v^+_{j+\frac12}+c^+_{j+\frac12},0\big\},
	\label{25}
\end{equation}
where $v^\pm_{j+\frac12}=(\rho v)^\pm_{j+\frac12}/\rho^\pm_{j+\frac12}$,
$c^\pm_{j+\frac12}=\sqrt{\gamma p^\pm_{j+\frac12}/\rho^\pm_{j+\frac12}}$, and $p^\pm_{j+\frac12}=p\big(\bU^\pm_{j+\frac12}\big)$. In fact,
$\sigma^\pm_{j+\frac12}$ defined in \eqref{25} and corrected according to Remark \ref{rem21}, satisfy both the BP Conditions \#2 and \#3. In
order to prove this, we adopt the GQL approach \cite{Wu2021GQL}.

First, recall that according to \cite[Theorem 5.1]{Wu2021GQL}, the equivalent GQL representation of the invariant region \eqref{22} is given
by
$$
{\G}_*:=\left\{\,\bU\in\mathbb R^3: \bU\cdot\bm e_1>0,\,\bU_1\cdot\bm n_*>0~\forall v_*\in\mathbb R\,\right\},
$$
where $\bm e_1=(1,0,0)^\top$ and $\bm n_*=\big(\frac{v_*^2}{2},-v_*,1\big)^\top$.

We then prove the following two auxiliary lemmas.
\begin{lemma}\label{lem28}
	For any $\bU\in{\G}$, we have $\left[c\bU\pm\big(v\bU-\bm f(\bU)\big)\right]\cdot\bm n>0$, where $\bm n=\bm e_1$ or $\bm n=\bm n_*$ for
	any $v_*\in\mathbb R$.
\end{lemma}
\begin{proof}
	If $\bm n=\bm e_1$, then $\left[c\bU\pm\big(v\bU-\bm f(\bU)\big)\right]\cdot\bm e_1=c\rho>0$. If $\bm n=\bm n_*$, then
	$$
	c\,\bU\cdot\bm n_*=\sqrt{\frac{\gamma p}{\rho}}\left(\frac{\rho}{2}(v-v_*)^2+\frac{p}{\gamma-1}\right)>
	\sqrt{\frac{(\gamma-1)p}{2\rho}}\sqrt{2\rho(v-v_*)^2\frac{p}{\gamma-1}}=p\,|v-v_*|,
	$$
	which implies
	$$
	\left[c\bU\pm\big(v\bU-\bm f(\bU)\big)\right]\cdot\bm n_*>p\,|v-v_*|\mp(0,p,pv)^\top\cdot\bm n_*=p\,|v-v_*|\mp p(v-v_*)\ge0~~\forall
	v_*\in\mathbb R.
	$$
	The proof is completed.
\end{proof}
\begin{lemma}\label{lem29}
	For any $\bU^+,\bU^-\in{\G}$, if $\sigma^+\ge v^++c^+$, $\sigma^-\le v^--c^-$, and $\sigma^+>\sigma^-$, then
	$R_f(\sigma^+,\sigma^-,\bU^+,\bU^-)\in{\G}$.
\end{lemma}
\begin{proof}
	For $\bm n=\bm e_1$ or $\bm n=\bm n_*$, \cite[Theorem 5.1]{Wu2021GQL} yields $\bU^+\cdot\bm n>0$ and $\bU^-\cdot\bm n>0$. \Cref{lem28}
	yields
	$$
	\begin{aligned}
		\left[\sigma^+\bU^+-\bm f(\bU^+)\right]\cdot\bm n&=\left[(v^++c^+)\,\bU^+-\bm f(\bU^+)\right]\cdot\bm n+\left[\sigma^+-(v^++c^+)\right]
		(\bU^+\cdot\bm n)\\
		&\ge\left[c^+\bU^++\left(v^+\bU^+-\bm f(\bU^+)\right)\right]\cdot\bm n>0,\\
		\left[\bm f(\bU^-)-\sigma^-\bU^-\right]\cdot\bm n&=\left[\bm f(\bU^-)-(v^--c^-)\,\bU^-\right]\cdot\bm n+\left[(v^--c^+)-\sigma^-\right]
		(\bU^-\cdot\bm n)\\
		&\ge\left[c^-\bU^--\left(v^-\bU^--\bm f(\bU^-)\right)\right]\cdot\bm n>0.
	\end{aligned}
	$$
	Therefore,
	$$
	R_f(\sigma^+,\sigma^-,\bU^+,\bU^-)\cdot\bm n\stackrel{\eqref{8}}{=}\frac{\left[\bm f(\bU^-)-\sigma^-\bU^-\right]\cdot\bm n+
		\left[\sigma^+\bU^+-\bm f(\bU^+)\right]\cdot\bm n}{\sigma^+-\sigma^-}>0.
	$$
	Hence, according to \cite[Theorem 5.1]{Wu2021GQL}, $R_f(\sigma^+,\sigma^-,\bU^+,\bU^-)\in{\G}_*={\G}$, and the proof is completed.
\end{proof}

We now state and prove the main result of this subsection.
\begin{lemma}\label{lem210}
	For the 1-D Euler equations of gas dynamics \eqref{1f}, \eqref{20} with the invariant region \eqref{22}, the speeds $\sigma^\pm_{j+\frac12}$
	defined in \eqref{25} and corrected according to Remark \ref{rem21}, satisfy both the BP Conditions \#2 and \#3.
\end{lemma}
\begin{proof}
	We note that \eqref{25} together with the correction introduced in Remark \ref{rem21} imply $\,\sigma^+_{j+\frac12}\ge v^+_{j+\frac12}+c^+_{j+\frac12}$,
	$\,\sigma^-_{j+\frac12}\le v^-_{j+\frac12}-c^-_{j+\frac12}$, and $\,\sigma^+_{j+\frac12}>\sigma^-_{j+\frac12}$, and then we use \Cref{lem29}
	to obtain $\bU^*_{j+\frac12}=R_f\big(\sigma^+_{j+\frac12},\sigma^-_{j+\frac12},\bU^+_{j+\frac12},\bU^-_{j+\frac12}\big)\in{\G}$, which
	means that the BP Condition \#2 is satisfied. Similarly, because $\,\sigma^+_{j+\frac12}\ge v^-_{j+\frac12}+c^-_{j+\frac12}$,
	$\,\sigma^-_{j-\frac12}\le v^+_{j-\frac12}-c^+_{j-\frac12}$, and $\,\sigma^+_{j+\frac12}>\sigma^-_{j-\frac12}$, \Cref{lem29} yields\\
	$\bU^*_j=R_f\big(\sigma^+_{j+\frac12},\sigma^-_{j-\frac12},\bU^-_{j+\frac12},\bU^+_{j-\frac12}\big)\in{\G}$, so that the BP Condition
	\#3 is satisfied. The proof is thus completed.
\end{proof}

\subsubsection{Enforcing the BP Condition \#4}\label{sec233}
Our numerical experiments (not shown in this paper) clearly demonstrate that for the Euler equations of gas dynamics, the use of the
anti-diffusion \eqref{7} may lead to appearance of negative pressures. In order to address this issue, we replace the anti-diffusion term
$\bd_{j+\frac12}$ with
\begin{equation}
	\begin{aligned}
		&\widetilde\bd_{j+\frac12}=\beta_{j+\frac12}\bd_{j+\frac12}=
		\beta_{j+\frac12}\,{\rm minmod}\,\big(\bU^+_{j+\frac12}-\bU^*_{j+\frac12},\bU^*_{j+\frac12}-\bU^-_{j+\frac12}\big),\\
		&\beta_{j+\frac12}:=\min\bigg\{\frac{p\big(\bU^*_{j+\frac12}\big)-\hat\varepsilon_{j+\frac12}}
		{p\big(\bU^*_{j+\frac12}\big)-p_{j+\frac12}^{\min}},1\bigg\},\\
		&\hat\varepsilon_{j+\frac12}:=\min\big\{10^{-13},p\big(\bU^*_{j+\frac12}\big)\big\},\quad p_{j+\frac12}^{\min}:=
		\min\big\{p\big(\bU^{*,+}_{j+\frac12}\big),p\big(\bU^{*,-}_{j+\frac12}\big)\big\},
	\end{aligned}
	\label{26}
\end{equation}
where $\bU^{*,\pm}_{j+\frac12}$ are defined in \eqref{9}.
\begin{lemma}\label{lem211}
	For the 1-D Euler equations of gas dynamics \eqref{1f}, \eqref{20} with the invariant region \eqref{22}, the modified quantity
	$$
	\tilde\bU^{*,\pm}_{j+\frac12}:=
	\bU^*_{j+\frac12}-\frac{\sigma^\pm_{j+\frac12}}{\sigma^+_{j+\frac12}-\sigma^-_{j+\frac12}}\widetilde\bd_{j+\frac12}
	$$
	satisfies the BP Conditions \#4.
\end{lemma}
\begin{proof}
	First, we use \eqref{26} and \eqref{9} to rewrite $\tilde\bU^{*,\pm}_{j+\frac12}$ as
	$$
	\begin{aligned}
		\tilde\bU^{*,\pm}_{j+\frac12}&=\bU^*_{j+\frac12}-\beta_{j+\frac12}\frac{\sigma^\pm_{j+\frac12}}{\sigma^+_{j+\frac12}-\sigma^-_{j+\frac12}}
		\,{\rm minmod}\,\big(\bU^+_{j+\frac12}-\bU^*_{j+\frac12},\bU^*_{j+\frac12}-\bU^-_{j+\frac12}\big)\\
		&=\bU^*_{j+\frac12}-\beta_{j+\frac12}\big(\bU^*_{j+\frac12}-\bU^{*,\pm}_{j+\frac12}\big)=\big(1-\beta_{j+\frac12}\big)\bU^*_{j+\frac12}+
		\beta_{j+\frac12}\bU^{*,\pm}_{j+\frac12}.
	\end{aligned}
	$$
	Since $\rho^\pm_{j+\frac12}>0$ and $\rho^*_{j+\frac12}>0$, in order to show that $\tilde\rho^{*,\pm}_{j+\frac12}>0$, one has to show that
	$\rho^{*,\pm}_{j+\frac12}>0$. Indeed, if $\rho^+_{j+\frac12}>\rho^*_{j+\frac12}>\rho^-_{j+\frac12}$, then
	$$
	\begin{aligned}
		\rho^{*,+}_{j+\frac12}&=\rho^*_{j+\frac12}-\frac{\sigma^+_{j+\frac12}}{\sigma^+_{j+\frac12}-\sigma^-_{j+\frac12}}
		\,{\rm minmod}\,\big(\rho^+_{j+\frac12}-\rho^*_{j+\frac12},\rho^*_{j+\frac12}-\rho^-_{j+\frac12}\big)
		\ge\rho^*_{j+\frac12}-\big(\rho^*_{j+\frac12}-\rho^-_{j+\frac12}\big)=\rho^-_{j+\frac12}>0,
	\end{aligned}
	$$
	and otherwise, $\rho^{*,+}_{j+\frac12}=\rho^*_{j+\frac12}>0$. Similarly, if $\rho^+_{j+\frac12}<\rho^*_{j+\frac12}<\rho^-_{j+\frac12}$, then
	$$
	\begin{aligned}
		\rho^{*,-}_{j+\frac12}&=\rho^*_{j+\frac12}-\frac{\sigma^-_{j+\frac12}}{\sigma^+_{j+\frac12}-\sigma^-_{j+\frac12}}
		\,{\rm minmod}\,\big(\rho^+_{j+\frac12}-\rho^*_{j+\frac12},\rho^*_{j+\frac12}-\rho^-_{j+\frac12}\big)
		\ge\rho^*_{j+\frac12}+
		\big(\rho^+_{j+\frac12}-\rho^*_{j+\frac12}\big)=\rho^+_{j+\frac12}>0,
	\end{aligned}
	$$
	and otherwise, $\rho^{*,-}_{j+\frac12}=\rho^*_{j+\frac12}>0$.
	
	In order to prove that $p\big(\tilde\bU^{*,\pm}_{j+\frac12}\big)>0$, we consider two possible cases. If
	$p_{j+\frac12}^{\min}\ge\hat\varepsilon_{j+\frac12}$, \eqref{26} yields $\beta_{j+\frac12}=1$ and thus
	$p\big(\tilde\bU^{*,\pm}_{j+\frac12}\big)=p\big(\bU^{*,\pm}_{j+\frac12}\big)\ge p_{j+\frac12}^{\min}\ge\hat\varepsilon_j>0$. Otherwise, we
	apply Jensen's inequality to the pressure function $p(\bU)$, which is concave when $\rho>0$, and obtain
	$$
	\begin{aligned}
		p\big(\tilde\bU^{*,\pm}_{j+\frac12}\big)&=p\big(\big(1-\beta_{j+\frac12}\big)\,\bU_{j+\frac12}^*+\beta_{j+\frac12}\bU^{*,\pm}_{j+\frac12}
		\big)\ge(1-\beta_{j+\frac12})p\big(\bU_{j+\frac12}^*\big)+\beta_{j+\frac12}p\big(\bU^{*,\pm}_{j+\frac12}\big)\\
		&\stackrel{\eqref{26}}{=}\bigg(1-\frac{p\big(\bU^*_{j+\frac12}\big)-\hat\varepsilon_{j+\frac12}}
		{p\big(\bU^*_{j+\frac12}\big)-p_{j+\frac12}^{\min}}\bigg)p\big(\bU_{j+\frac12}^*\big)+
		\frac{p\big(\bU^*_{j+\frac12}\big)-\hat\varepsilon_{j+\frac12}}{p\big(\bU^*_{j+\frac12}\big)-p_{j+\frac12}^{\min}}\,
		p\big(\bU^{*,\pm}_{j+\frac12}\big)\\
		&\ge\frac{\hat\varepsilon_{j+\frac12}-p_{j+\frac12}^{\min}}{p\big(\bU^*_{j+\frac12}\big)-p_{j+\frac12}^{\min}}\,p\big(\bU_{j+\frac12}^*\big)
		+\frac{p\big(\bU^*_{j+\frac12}\big)-\hat\varepsilon_{j+\frac12}}{p\big(\bU^*_{j+\frac12}\big)-p_{j+\frac12}^{\min}}\,
		p_{j+\frac12}^{\min}=\hat\varepsilon_{j+\frac12}>0.
	\end{aligned}
	$$
	The proof is completed.
\end{proof}

\subsubsection{Constructing BPCU Schemes}
Based on \Cref{lem27,lem210,lem211}, the BPCU schemes for the 1-D Euler equations of gas dynamics \eqref{1f}, \eqref{20} can be constructed
via the following steps.

\noindent
{\bf Step 1.} Modify the reconstructed point values $\bU_{j+\frac12}^\pm$ using \eqref{23}.

\noindent
{\bf Step 2.} Modify the ``built-in'' anti-diffusion term $\bd_{j+\frac12}$ using \eqref{26}.

\section{Two-Dimensional BPCU Schemes}\label{sec3}
The 1-D BP framework for CU
schemes proposed in \S\ref{sec22} can be extended to multiple dimensions. Without loss of generality, we present the 2-D BP framework and
2-D BPCU schemes, while the extensions to higher dimensions are similar and thus omitted.

Consider the 2-D hyperbolic system of conservation laws
\begin{equation}
	\bU_t+\bF(\bU)_x+\bG(\bU)_y=\bm0.
	\label{28}
\end{equation}
Assume that the computational mesh is uniform with Cartesian cells
$\Omega_{j,k}=[x_{j-\frac12},x_{j+\frac12}]\times[y_{k-\frac12},y_{k+\frac12}]$. Let $\,\xbar\bU_{j,k}^{\,n}$ be the approximation of the
cell average of $\bU(x,y,t^n)$ over $\Omega_{j,k}$ at time $t=t^n$. For the sake of convenience, we focus on the first-order forward Euler
time discretization, while all of the results below are also applicable to higher-order SSP time discretization.

The 2-D CU schemes read as
\begin{equation}
	\xbar\bU^{\,n+1}_{j,k}=\,\xbar\bU^{\,n}_{j,k}-\lambda^n\left(\hat\bF_{j+\frac12,k}-\hat\bF_{j-\frac12,k}\right)
	-\mu^n\left(\hat\bG_{j,k+\frac12}-\hat\bG_{j,k-\frac12}\right),\quad\mu^n:=\frac{\dt^n}{\dy},
	\label{29}
\end{equation}
with $\lambda^n$ and $\dt^n$ given by \eqref{2} and the CU numerical fluxes
\begin{equation}
	\begin{aligned}
		\hat\bF_{j+\frac12,k}&=\frac{\sigma^+_{j+\frac12,k}\bF\big(\bU^-_{j+\frac12,k}\big)-\sigma^-_{j+\frac12,k}\bF\big(\bU^+_{j+\frac12,k}\big)}
		{\sigma^+_{j+\frac12,k}-\sigma^-_{j+\frac12,k}}
		+\frac{\sigma^+_{j+\frac12,k}\sigma^-_{j+\frac12,k}}
		{\sigma^+_{j+\frac12,k}-\sigma^-_{j+\frac12,k}}\big(\bU^+_{j+\frac12,k}-\bU^-_{j+\frac12,k}-\bd_{j+\frac12,k}\big),\\
		\hat\bG_{j,k+\frac12}&=\frac{\sigma^+_{j,k+\frac12}\bG\big(\bU^-_{j,k+\frac12}\big)-\sigma^-_{j,k+\frac12}\bG\big(\bU^+_{j,k+\frac12}\big)}
		{\sigma^+_{j,k+\frac12}-\sigma^-_{j,k+\frac12}}
		+\frac{\sigma^+_{j,k+\frac12}\sigma^-_{j,k+\frac12}}
		{\sigma^+_{j,k+\frac12}-\sigma^-_{j,k+\frac12}}\big(\bU^+_{j,k+\frac12}-\bU^-_{j,k+\frac12}-\bd_{j,k+\frac12}\big).
	\end{aligned}
	\label{30}
\end{equation}
In \eqref{30}, $\bU^\pm_{j+\frac12,k}$ and $\bU^\pm_{j,k+\frac12}$ are the one-sided values of the reconstructed solution at the midpoints
of the cell interfaces $(x,y)=(x_{j+\frac12},y_k)$ and $(x,y)=(x_j,y_{k+\frac12})$, respectively:
\begin{equation}
	\begin{aligned}
		&\bU^-_{j+\frac12,k}=\,\xbar{\bU}_{j,k}^{\,n}+\frac{\dx}{2}(\bU_x)_{j,k}^n,~~
		\bU^+_{j+\frac12,k}=\,\xbar{\bU}_{j+1,k}^{\,n}-\frac{\dx}{2}(\bU_x)_{j+1,k}^n,\\
		&\bU^-_{j,k+\frac12}=\,\xbar{\bU}_{j,k}^{\,n}+\frac{\dy}{2}(\bU_y)_{j,k}^n,~~
		\bU^+_{j,k+\frac12}=\,\xbar{\bU}_{j+1,k}^{\,n}-\frac{\dy}{2}(\bU_y)_{j,k+1}^n,
	\end{aligned}
	\label{31}
\end{equation}
where
\begin{equation}
	\begin{aligned}
		&(\bU_x)_{j,k}^n={\rm minmod}\left(\theta\,\frac{\xbar{\bU}_{j,k}^{\,n}-\,\xbar{\bU}_{j-1,k}^{\,n}}{\dx},
		\frac{\xbar{\bU}_{j+1,k}^{\,n}-\,\xbar{\bU}_{j-1,k}^{\,n}}{2\dx},\theta\,\frac{\xbar{\bU}_{j+1,k}^{\,n}-\,\xbar{\bU}_{j,k}^{\,n}}{\dx}
		\right),\\
		&(\bU_y)_{j,k}^n={\rm minmod}\left(\theta\,\frac{\xbar{\bU}_{j,k}^{\,n}-\,\xbar{\bU}_{j,k-1}^{\,n}}{\dy},
		\frac{\xbar{\bU}_{j,k+1}^{\,n}-\,\xbar{\bU}_{j,k-1}^{\,n}}{2\dy},\theta\,\frac{\xbar{\bU}_{j,k+1}^{\,n}-\,\xbar{\bU}_{j,k}^{\,n}}{\dy}
		\right),
	\end{aligned}\quad\theta\in[1,2].
	\label{32}
\end{equation}
The quantities $\sigma^\pm_{j+\frac12,k}$ and $\sigma^\pm_{j,k+\frac12}$ in \eqref{30} denote the one-sided local speeds of propagation in
the $x$- and $y$-directions, respectively, and they can be estimated using the smallest and largest eigenvalues of the Jacobians
$A(\bU):=\partial\bF/\partial\bU$ and $B(\bU):=\partial\bG/\partial\bU$:
\begin{equation}
	\begin{aligned}
		\sigma^-_{j+\frac12,k}&=\min\big\{\lambda_1\big(A\big(\bU^-_{j+\frac12,k}\big)\big),
		\lambda_1\big(A\big(\bU^+_{j+\frac12,k}\big)\big),0\big\},\\
		\sigma^+_{j+\frac12,k}&=\max\big\{\lambda_m\big(A\big(\bU^-_{j+\frac12,k}\big)\big),
		\lambda_m\big(A\big(\bU^+_{j+\frac12,k}\big)\big),0\big\},\\
		\sigma^-_{j,k+\frac12}&=\min\big\{\lambda_1\big(B\big(\bU^-_{j,k+\frac12}\big)\big),
		\lambda_1\big(B\big(\bU^+_{j,k+\frac12}\big)\big),0\big\},\\
		\sigma^+_{j,k+\frac12}&=\max\big\{\lambda_m\big(B\big(\bU^-_{j,k+\frac12}\big)\big),
		\lambda_m\big(B\big(\bU^+_{j,k+\frac12}\big)\big),0\big\}.
	\end{aligned}
	\label{33}
\end{equation}
As in the 1-D case, \eqref{33} should be modified as in \Cref{rem21} to desingularize the computations in \eqref{30}.

Finally, $\bd_{j+\frac12,k}$ and $\bd_{j,k+\frac12}$ in \eqref{30} are the ``built-in'' anti-diffusion terms given by (see
\cite{ChertockCuiKurganovOzcanTadmor2018})
\begin{equation}
	\begin{aligned}
		&\bd_{j+\frac12,k}={\rm minmod}\,\big(\bU^+_{j+\frac12,k}-\bU^*_{j+\frac12,k},\bU^*_{j+\frac12,k}-\bU^-_{j+\frac12,k}\big),\\
		&\bd_{j,k+\frac12}={\rm minmod}\,\big(\bU^+_{j,k+\frac12}-\bU^*_{j,k+\frac12},\bU^*_{j,k+\frac12}-\bU^-_{j,k+\frac12}\big),
	\end{aligned}
	\label{34}
\end{equation}
where
\begin{equation}
	\begin{aligned}
		&\bU^*_{j+\frac12,k}=R_f\big(\sigma^+_{j+\frac12,k},\sigma^-_{j+\frac12,k},\bU^+_{j+\frac12,k},\bU^-_{j+\frac12,k}\big),\\
		&R_f(\sigma^+,\sigma^-,\bU^+,\bU^-):=\frac{\sigma^+\bU^+-\sigma^-\bU^--\bF(\bU^+)+\bF(\bU^-)}{\sigma^+-\sigma^-},\\
		&\bU^*_{j,k+\frac12}=R_g\big(\sigma^+_{j,k+\frac12},\sigma^-_{j,k+\frac12},\bU^+_{j,k+\frac12},\bU^-_{j,k+\frac12}\big),\\
		&R_g(\sigma^+,\sigma^-,\bU^+,\bU^-):=\frac{\sigma^+\bU^+-\sigma^-\bU^--\bG(\bU^+)+\bG(\bU^-)}{\sigma^+-\sigma^-}.
	\end{aligned}
	\label{35}
\end{equation}

The CU scheme \eqref{29}--\eqref{35} is, in general, not BP. In order to design a BPCU scheme, one may need to properly modify the
reconstructed point values ($\bU_{j+\frac12,k}^\pm$ and $\bU_{j,k+\frac12}^\pm$), one-sided local speeds of propagation
($\sigma_{j+\frac12,k}^\pm$ and $\sigma_{j,k+\frac12}^\pm$), and ``built-in'' anti-diffusion terms ($\bd_{j+\frac12,k}$ and
$\bd_{j,k+\frac12}$). To this end, we first introduce the following auxiliary quantities:
\begin{equation}
	\begin{aligned}
		&\bU^{*,\pm}_{j+\frac12,k}:=\bU^*_{j+\frac12,k}-\frac{\sigma^\pm_{j+\frac12,k}}{\sigma^+_{j+\frac12,k}-\sigma^-_{j+\frac12,k}}\,
		\bd_{j+\frac12,k},\\
		&\bU^{*,x}_{j,k}:=R_f\big(\sigma^+_{j+\frac12,k},\sigma^-_{j-\frac12,k},\bU^-_{j+\frac12,k},\bU^+_{j-\frac12,k}\big),\\
		&\bU^{*,\pm}_{j,k+\frac12}:=\bU^*_{j,k+\frac12}-\frac{\sigma^\pm_{j,k+\frac12}}{\sigma^+_{j,k+\frac12}-\sigma^-_{j,k+\frac12}}\,
		\bd_{j,k+\frac12},\\
		&\bU^{*,y}_{j,k}:=R_g\big(\sigma^+_{j,k+\frac12},\sigma^-_{j,k-\frac12},\bU^-_{j,k+\frac12},\bU^+_{j,k-\frac12}\big),
	\end{aligned}
	\label{36}
\end{equation}
and prove the following theorem, which is crucial for the development of 2-D BPCU schemes.
\begin{theorem}\label{thm31}
	The 2-D CU scheme \eqref{29}--\eqref{36} admits the following convex decomposition:
	\begin{equation}
		\begin{aligned}
			\xbar\bU^{\,n+1}_{j,k}&=\bigg(\frac{\lambda^n\alpha_1}{2\big(\lambda^n\alpha_1+\mu^n\alpha_2\big)}-\lambda^n\big(\sigma^+_{j-\frac12,k}-
			\sigma^-_{j-\frac12,k}\big)\bigg)\bU^+_{j-\frac12,k}+\lambda^n\sigma^+_{j-\frac12,k}\bU^{*,-}_{j-\frac12,k}\\
			&+\bigg(\frac{\lambda^n\alpha_1}{2\big(\lambda^n\alpha_1+\mu^n\alpha_2\big)}-\lambda^n\big(\sigma^+_{j+\frac12,k}-
			\sigma^-_{j+\frac12,k}\big)\bigg)\bU^-_{j+\frac12,k}-\lambda^n\sigma^-_{j+\frac12,k}\bU^{*,+}_{j+\frac12,k}
			+\lambda^n\big(\sigma^+_{j+\frac12,k}-\sigma^-_{j-\frac12,k}\big)\,\bU_{j,k}^{*,x}\\
			&+\bigg(\frac{\mu^n\alpha_2}{2\big(\lambda^n\alpha_1+\mu^n\alpha_2\big)}-\mu^n\big(\sigma^+_{j,k-\frac12}-\sigma^-_{j,k-\frac12}\big)\bigg)
			\bU^+_{j,k-\frac12}+\mu^n\sigma^+_{j,k-\frac12}\bU^{*,-}_{j,k-\frac12}\\
			&+\bigg(\frac{\mu^n\alpha_2}{2\big(\lambda^n\alpha_1+\mu^n\alpha_2\big)}-\mu^n\big(\sigma^+_{j,k+\frac12}-\sigma^-_{j,k+\frac12}\big)\bigg)
			\bU^-_{j,k+\frac12}-\mu^n\sigma^-_{j,k+\frac12,k}\bU^{*,+}_{j,k+\frac12}
			+\mu^n\big(\sigma^+_{j,k+\frac12}-\sigma^-_{j,k-\frac12}\big)\,\bU_{j,k}^{*,y}
		\end{aligned}
		\label{37}
	\end{equation}
	under the CFL condition
	\begin{equation}
		\lambda^n\alpha_1+\mu^n\alpha_2\le\frac12,\quad\alpha_1:=\max_{j,k}\big\{\sigma^+_{j+\frac12,k}-\sigma^-_{j+\frac12,k}\big\},\quad
		\alpha_2:=\max_{j,k}\big\{\sigma^+_{j,k+\frac12}-\sigma^-_{j,k+\frac12}\big\}.
		\label{38}
	\end{equation}
\end{theorem}
\begin{proof}
	Similarly to \eqref{12} and \eqref{13}, the fluxes $\hat\bF_{j\pm\frac12,k}$ and $\hat\bG_{j,k\pm\frac12}$ can be reformulated as
	$$
	\begin{aligned}
		&\hat\bF_{j+\frac12,k}=\bF\big(\bU^-_{j+\frac12,k}\big)-\sigma^-_{j+\frac12,k}\bU^-_{j+\frac12,k}+
		\sigma^-_{j+\frac12,k}\bU_{j+\frac12,k}^{*,+},\\
		&\hat\bF_{j-\frac12,k}=\bF\big(\bU^+_{j-\frac12,k}\big)
		-\sigma^+_{j-\frac12,k}\bU^+_{j-\frac12,k}+\sigma^+_{j-\frac12,k}\bU_{j-\frac12,k}^{*,-},\\
		&\hat\bG_{j,k+\frac12}=\bG\big(\bU^-_{j,k+\frac12}\big)-\sigma^-_{j,k+\frac12}\bU^-_{j,k+\frac12}+
		\sigma^-_{j,k+\frac12}\bU_{j,k+\frac12}^{*,+},\\
		&\hat\bG_{j,k-\frac12}=\bG\big(\bU^+_{j,k-\frac12}\big)
		-\sigma^+_{j,k-\frac12}\bU^+_{j,k-\frac12}+\sigma^+_{j,k-\frac12}\bU_{j,k-\frac12}^{*,-}.
	\end{aligned}
	$$
	Substituting them into \eqref{29} results in
	\begin{align*}
		&\xbar\bU^{\,n+1}_{j,k}=\frac{\lambda^n\alpha_1}{\lambda^n\alpha_1+\mu^n\alpha_2}\,\xbar\bU^{\,n}_{j,k}+
		\frac{\mu^n\alpha_2}{\lambda^n\alpha_1+\mu^n\alpha_2}\,\xbar\bU^{\,n}_{j,k}\\
		&\hspace*{0.3cm}-\lambda^n\hat\bF_{j+\frac12,k}+\lambda^n\hat\bF_{j-\frac12,k}-
		\mu^n\hat\bG_{j,k+\frac12}+\mu^n\hat\bG_{j,k-\frac12}\\
		&\hspace*{0.3cm}=\frac{\lambda^n\alpha_1}{2\big(\lambda^n\alpha_1+\mu^n\alpha_2\big)}\,\bU^-_{j+\frac12,k}-
		\lambda^n\big(\bF\big(\bU^-_{j+\frac12,k}\big)-\sigma^-_{j+\frac12,k}\bU^-_{j+\frac12,k}+\sigma^-_{j+\frac12,k}\bU_{j+\frac12,k}^{*,+}\big)\\
		&\hspace*{0.3cm}+\lambda^n\big(\sigma^+_{j+\frac12,k}-\sigma^+_{j+\frac12,k}\big)\,\bU^-_{j+\frac12,k}+
		\frac{\lambda^n\alpha_1}{2\big(\lambda^n\alpha_1+\mu^n\alpha_2\big)}\,\bU^+_{j-\frac12,k}\\
		&\hspace*{0.3cm}+\lambda^n\big(\bF\big(\bU^+_{j-\frac12,k}\big)-
		\sigma^+_{j-\frac12,k}\bU^+_{j-\frac12,k}+\sigma^+_{j-\frac12,k}\bU_{j-\frac12,k}^{*,-}\big)+\lambda^n\big(\sigma^-_{j-\frac12,k}-
		\sigma^-_{j-\frac12,k}\big)\,\bU^+_{j-\frac12,k}\\
		&\hspace*{0.3cm}+\frac{\mu^n\alpha_2}{2\big(\lambda^n\alpha_1+\mu^n\alpha_2\big)}\,\bU^-_{j,k+\frac12}-
		\mu^n\big(\bG\big(\bU^-_{j,k+\frac12}\big)-\sigma^-_{j,k+\frac12}\bU^-_{j,k+\frac12}+\sigma^-_{j,k+\frac12}\bU_{j,k+\frac12}^{*,+}\big)\\
		&\hspace*{0.3cm}+\mu^n\big(\sigma^+_{j,k+\frac12}-\sigma^+_{j,k+\frac12}\big)\,\bU^-_{j,k+\frac12}
		+\frac{\mu^n\alpha_2}{2\big(\lambda^n\alpha_1+\mu^n\alpha_2\big)}\,\bU^+_{j,k-\frac12}\\
		&\hspace*{0.3cm}+\mu^n\big(\bG\big(\bU^+_{j,k-\frac12}\big)-
		\sigma^+_{j,k-\frac12}\bU^+_{j,k-\frac12}+\sigma^+_{j,k-\frac12}\bU_{j,k-\frac12}^{*,-}\big)+\mu^n\big(\sigma^-_{j,k-\frac12}-
		\sigma^-_{j,k-\frac12}\big)\,\bU^+_{j,k-\frac12}\\
		&\hspace*{0.3cm}=\bigg[\frac{\lambda^n\alpha_1}{\lambda^n\alpha_1+\mu^n\alpha_2}-
		\lambda^n\big(\sigma_{j-\frac12,k}^+-\sigma_{j-\frac12,k}^-\big)\bigg]
		\bU_{j-\frac12,k}^++\lambda^n\sigma_{j-\frac12,k}^+\bU_{j-\frac12,k}^{*,-}\\
		&\hspace*{0.3cm}+\bigg[\frac{\lambda^n\alpha_1}{\lambda^n\alpha_1+\mu^n\alpha_2}-
		\lambda^n\big(\sigma_{j+\frac12,k}^+-\sigma_{j+\frac12,k}^-\big)\bigg]
		\bU_{j+\frac12,k}^--\lambda^n\sigma_{j+\frac12,k}^-\bU_{j+\frac12,k}^{*,+}\\
		&\hspace*{0.3cm}+\lambda^n\big(\sigma_{j+\frac12,k}^+-\sigma_{j-\frac12,k}^-\big)\,
		\frac{\sigma_{j+\frac12,k}^+\bU_{j+\frac12,k}^--\sigma_{j-\frac12,k}^-\bU_{j-\frac12,k}^+-\bF\big(\bU_{j+\frac12,k}^-\big)+
			\bF\big(\bU_{j-\frac12,k}^+\big)}{\sigma_{j+\frac12,k}^+-\sigma_{j-\frac12,k}^-}\\ 
		&\hspace*{0.3cm}+\bigg[\frac{\mu^n\alpha_2}{\lambda^n\alpha_1+\mu^n\alpha_2}-
		\mu^n\big(\sigma_{j,k-\frac12}^+-\sigma_{j,k-\frac12}^-\big)\bigg]
		\bU_{j,k-\frac12}^++\mu^n\sigma_{j,k-\frac12}^+\bU_{j,k-\frac12}^{*,-}\\
		&\hspace*{0.3cm}+\bigg[\frac{\mu^n\alpha_2}{\lambda^n\alpha_1+\mu^n\alpha_2}-
		\mu^n\big(\sigma_{j,k+\frac12}^+-\sigma_{j,k+\frac12}^-\big)\bigg]
		\bU_{j,k+\frac12}^--\mu^n\sigma_{j,k+\frac12}^-\bU_{j,k+\frac12}^{*,+}\\
		&\hspace*{0.3cm}+\mu^n\big(\sigma_{j,k+\frac12}^+-\sigma_{j,k-\frac12}^-\big)\,
		\frac{\sigma_{j,k+\frac12}^+\bU_{j,k+\frac12}^--\sigma_{j,k-\frac12}^-\bU_{j,k-\frac12}^+-\bG\big(\bU_{j,k+\frac12}^-\big)+
			\bG\big(\bU_{j,k-\frac12}^+\big)}{\sigma_{j,k+\frac12}^+-\sigma_{j,k-\frac12}^-},
	\end{align*}
	which immediately implies the decomposition \eqref{37}. Under the CFL condition \eqref{38}, the decomposition \eqref{37} recasts
	$\,\xbar\bU_{j,k}^{\,n+1}$ as a convex combination of $\bU^\mp_{j\pm\frac12,k}$, $\bU^\mp_{j,k\pm\frac12}$, $\bU^{*,\pm}_{j\pm\frac12,k}$,
	$\bU^{*,\pm}_{j,k\pm\frac12}$, $\bU^{*,x}_{j,k}$, and $\bU^{*,y}_{j,k}$. 
\end{proof}

Thanks to the convexity of ${\G}$, \Cref{thm31} immediately leads to the following sufficient condition for obtaining 2-D BPCU schemes.
\begin{corollary}\label{cor31}
	If
	\begin{equation}
		\bU^\pm_{j+\frac12,k}\in{\G},\quad\bU^\pm_{j,k+\frac12}\in{\G},\quad\bU^{*,\pm}_{j+\frac12,k}\in{\G},\quad
		\bU^{*,\pm}_{j,k+\frac12}\in{\G},\quad\bU^{*,x}_{j,k}\in{\G},\quad\bU^{*,y}_{j,k}\in{\G},
		\label{39f}
	\end{equation}
	for all $j,k$, then $\,\xbar\bU_{j,k}^{\,n+1}\in{\G}$ for all $j,k$ and the CU scheme \eqref{29}--\eqref{30} is BP under the CFL
	condition \eqref{38}.
\end{corollary}

Inspired by \Cref{cor31}, we simplify the challenging goal of designing 2-D BPCU schemes into four more accessible tasks, that is, modifying
the 2-D CU schemes \eqref{29}--\eqref{35} in such a way that the following four essential conditions are satisfied:
\begin{itemize}[leftmargin=43mm]
	\setlength\itemsep{0mm}
	\item[\bf 2-D BP Condition \#1:] If $\,\xbar\bU_{j,k}^{\,n}\in{\G}$ for all $j,k$, then
	$\bU^\pm_{j+\frac12,k},\bU^\pm_{j,k+\frac12}\in{\G}$ for all $j,k$;
	\item[\bf 2-D BP Condition \#2:] If $\bU^\pm_{j+\frac12,k},\bU^\pm_{j,k+\frac12}\in{\G}$ for all $j,k$, then
	$\bU^*_{j+\frac12,k},\bU^*_{j,k+\frac12}\in{\G}$ for all $j,k$;
	\item[\bf 2-D BP Condition \#3:] If $\bU^\pm_{j+\frac12,k},\bU^\pm_{j,k+\frac12}\in{\G}$ for all $j,k$, then
	$\bU^{*,x}_{j,k},\bU^{*,y}_{j,k}\in{\G}$ for all $j,k$;
	\item[\bf 2-D BP Condition \#4:] If $\bU^\pm_{j+\frac12,k},\bU^\pm_{j,k+\frac12},\bU^*_{j+\frac12,k},\bU^*_{j,k+\frac12}\in{\G}$ for all
	$j,k$, then $\bU^{*,\pm}_{j+\frac12,k},\bU^{*,\pm}_{j,k+\frac12}\in{\G}$ for all $j,k$.
\end{itemize}

In \S\ref{sec31}, we will follow the above BP framework and design provably BPCU for the 2-D Euler equations of gas dynamic.

\subsection{BPCU schemes for 2-D  compressible Euler equations}\label{sec31}
The 2-D Euler equations of gas dynamics for ideal gases read as \eqref{28} with
\begin{equation}
	\begin{aligned}
		&\bU=\big(\rho,\rho v_1,\rho v_2,E\big)^\top,\quad\bF(\bU)=\big(\rho v_1,\rho v_1^2+p,\rho v_1v_2,(E+p)v_1\big)^\top,\\
		&\bG(\bU)=\big(\rho v_2,\rho v_1v_2,\rho v_2^2+p,(E+p)v_2\big)^\top,\quad E=\frac12\rho(v_1^2+v_2^2)+\frac{p}{\gamma-1}.
	\end{aligned}
	\label{39}
\end{equation}
Here, $v_1$ and $v_2$ are the $x$- and $y$-velocities, respectively, and the rest of notations are the same as in \S\ref{sec23}. The
eigenvalues of the Jacobians $A=\partial\bF/\partial\bU$ and $B=\partial\bG/\partial\bU$ are 
$$
\begin{aligned}
	&\lambda_1\big(A(\bU)\big)=v_1-c,\quad\lambda_2\big(A(\bU)\big)=\lambda_3\big(A(\bU)\big)=v_1,\quad\lambda_4\big(A(\bU)\big)=v_1+c,\\
	&\lambda_1\big(B(\bU)\big)=v_2-c,\quad\lambda_2\big(B(\bU)\big)=\lambda_3\big(B(\bU)\big)=v_2,\quad\lambda_4\big(B(\bU)\big)=v_2+c,
\end{aligned}
$$
The set of admissile states,
\begin{equation}
	{\G}=\Big\{\,\bU\in\mathbb R^4: \rho>0,~p(\bU)=(\gamma-1)\Big(E-\frac12\rho(v_1^2+v_2^2)\Big)>0\,\Big\},
	\label{41}
\end{equation}
is a convex invariant region for the 2-D Euler equations of gas dynamics since the pressure function $p(\bU)$ is concave with respect to
$\bU$ when $\rho>0$.

The key components in the construction of the 2-D BPCU scheme for the 2-D Euler equations of gas dynamics are presented below. First, the
limited one-sided point values $\bU^\mp_{j\pm\frac12,k}$ and $\bU^\mp_{j,k\pm\frac12}$ defined in \eqref{31}--\eqref{32}, are replaced with
\begin{equation}
	\tilde\bU^\mp_{j\pm\frac12,k}=\big(1-\delta^x_{j,k}\big)\,\xbar\bU_{j,k}^{\,n}+\delta^x_{j,k}\,\bU^\mp_{j\pm\frac12,k},\quad
	\tilde\bU^\mp_{j,k\pm\frac12}=\big(1-\delta^y_{j,k}\big)\,\xbar\bU_{j,k}^{\,n}+\delta^y_{j,k}\,\bU^\mp_{j,k\pm\frac12},
	\label{42}
\end{equation}
where
\begin{equation}
	\begin{aligned}
		&\delta^x_{j,k}:=\min\bigg\{\frac{p(\,\xbar\bU_{j,k}^{\,n})-\hat\varepsilon_{j,k}}
		{p(\,\xbar\bU_{j,k}^{\,n})-\hat p^x_{j,k}},1\bigg\},\quad
		\delta^y_{j,k}:=\min\bigg\{\frac{p(\,\xbar\bU_{j,k}^{\,n})-\hat\varepsilon_{j,k}}{p(\,\xbar\bU_{j,k}^{\,n})-\hat p^y_{j,k}},1\bigg\},\\
		&\hat p^x_{j,k}:=\min\big\{p\big(\bU^+_{j-\frac12,k}\big),p\big(\bU^-_{j+\frac12,k}\big)\big\},\quad
		\hat p^y_{j,k}:=\min\big\{p\big(\bU^+_{j,k-\frac12}\big),p\big(\bU^-_{j,k+\frac12}\big)\big\},\\
		&\hat\varepsilon_{j,k}:=\min\big\{10^{-13},p(\,\xbar\bU_{j,k}^{\,n})\big\}.
	\end{aligned}
	\label{43}
\end{equation}
It can be verified that the values $\tilde\bU^\pm_{j+\frac12,k}$ and $\tilde\bU^\pm_{j,k+\frac12}$ defined in \eqref{42}--\eqref{43} belong
to ${\G}$.

The corrected reconstructed values are then used to evaluate the one-sided local speeds of propagation:
\begin{equation}
	\begin{aligned}
		&\sigma^-_{j+\frac12,k}=\min\big\{(v_1)^-_{j+\frac12,k}-c^-_{j+\frac12,k},(v_1)^+_{j+\frac12,k}-c^+_{j+\frac12,k},0\big\},\\
		&\sigma^+_{j+\frac12,k}=\max\big\{(v_1)^-_{j+\frac12,k}+c^-_{j+\frac12,k},(v_1)^+_{j+\frac12,k}+c^+_{j+\frac12,k},0\big\},\\
		&\sigma^-_{j,k+\frac12}=\min\big\{(v_2)^-_{j,k+\frac12}-c^-_{j,k+\frac12},(v_2)^+_{j,k+\frac12}-c^+_{j,k+\frac12},0\big\},\\
		&\sigma^+_{j,k+\frac12}=\max\big\{(v_2)^-_{j,k+\frac12}+c^-_{j,k+\frac12},(v_2)^+_{j,k+\frac12}+c^+_{j,k+\frac12},0\big\}.
	\end{aligned}
	\label{44}
\end{equation}

Finally, the ``build-in'' anti-diffusion terms $\bd_{j+\frac12,k}$ and $\bd_{j,k+\frac12}$ defined in \eqref{34}, are replaced with
\begin{equation}
	\tilde\bd_{j+\frac12,k}=\beta_{j+\frac12,k}\bd_{j+\frac12,k},\quad\tilde\bd_{j,k+\frac12}=\beta_{j,k+\frac12}\bd_{j,k+\frac12},
	\label{45}
\end{equation}
where
\begin{equation}
	\begin{aligned}
		&\beta_{j+\frac12,k}:=\min\bigg\{\frac{p\big(\bU^*_{j+\frac12,k}\big)-\hat\varepsilon_{j+\frac12,k}}
		{p\big(\bU^*_{j+\frac12,k}\big)-\hat p_{j+\frac12,k}},1\bigg\},&&
		\beta_{j,k+\frac12}:=\min\bigg\{\frac{p\big(\bU^*_{j,k+\frac12}\big)-\hat\varepsilon_{j,k+\frac12}}
		{p\big(\bU^*_{j,k+\frac12}\big)-\hat p_{j,k+\frac12}},1\bigg\},\\
		&\hat\varepsilon_{j+\frac12,k}:=\min\big\{10^{-13},p\big(\bU^*_{j+\frac12,k}\big)\big\},&&
		\hat\varepsilon_{j,k+\frac12}:=\min\big\{10^{-13},p\big(\bU^*_{j,k+\frac12}\big)\big\},\\
		&\hat p_{j+\frac12,k}:=\min\big\{p\big(\bU^{*,+}_{j+\frac12,k}\big),p\big(\bU^{*,-}_{j+\frac12,k}\big)\big\},&&
		\hat p_{j,k+\frac12}:=\min\big\{p\big(\bU^{*,+}_{j,k+\frac12}\big),p\big(\bU^{*,-}_{j,k+\frac12}\big)\big\},
	\end{aligned}
	\label{46}
\end{equation}
and $\bU_{j+\frac12,k}^{*,\pm}$ and $\bU_{j+\frac12,k}^{*,\pm}$ are defined in \eqref{36}.
\begin{theorem}\label{thm32}
	For the 2-D Euler equations of gas dynamics \eqref{28}, \eqref{39} with the invariant region \eqref{41}, the CU scheme
	\eqref{29}--\eqref{32}, \eqref{34}--\eqref{36} with the corrections \eqref{42}--\eqref{46} is BP under the CFL condition \eqref{38}.
\end{theorem}
\begin{proof}
	Similarly to \Cref{lem27}, it can be proven that the point values $\tilde\bU^\mp_{j\pm\frac12,k}$ and $\tilde\bU^\mp_{j,k\pm\frac12}$
	defined in \eqref{42}--\eqref{43} satisfy the 2-D BP Condition \#1. Analogously to \Cref{lem210}, one can show that the one-sided local
	speeds of propagation $\sigma^\pm_{j+\frac12,k}$ and $\sigma^\pm_{j,k+\frac12}$ given in \eqref{44} satisfy the 2-D BP Conditions \#2 and
	\#3. Similarly to \Cref{lem211}, one can prove that the ``built-in'' anti-diffusion terms $\tilde\bd_{j+\frac12,k}$ and
	$\tilde\bd_{j,k+\frac12}$ defined in \eqref{45}--\eqref{46} satisfy the 2-D BP Condition \#4. Thus, all of the the conditions in \eqref{39f}
	are satisfied, which, according to \Cref{thm31}, the CU schemes \eqref{29}--\eqref{32}, \eqref{34}--\eqref{36} with the corrections
	\eqref{42}--\eqref{46} are BP under the CFL condition \eqref{38}.
\end{proof}

\section{Numerical Experiments}\label{sec4}
In this section, we conduct several numerical experiments to examine the accuracy and robustness of the proposed BPCU schemes, as well as
its capability of preserving the invariant region ${\G}$. The performance of BPCU schemes will be compared with the performance of the
following alternative methods:

\noindent
$\bullet$ {\em Scheme 1:} the original CU scheme;

\noindent
$\bullet$ {\em Scheme 2:} the BPCU scheme with the anti-diffusion term switched off by setting $\beta=0$ in \eqref{26} and \eqref{45};

\noindent
$\bullet$ {\em Scheme 3:} the second-order semi-discrete finite-volume scheme with a global Lax-Friedrichs numerical flux, which is obtained
from Scheme 2 by replacing the one-sided local speeds of propagation $\sigma^\pm$ with $\pm\sigma_{\max}$, where $\sigma_{\max}$ denotes the
global maximum characteristic speed over the entire computational domain.

\smallskip
We take the minmod limiter parameter $\theta=1.3$ and employ the two-stage second-order explicit SSP Runge-Kutta method (the Heun method)
\cite{GottliebKetchesonShu2011,GST} for time discretization in all of the reported numerical experiments. The specific heat ration is
either $\gamma=1.4$ (\Cref{ex41,ex42,ex43,ex46,ex47}) or $\gamma=5/3$ (\Cref{ex44,ex45}).

\begin{example}\label{ex41}
	In the first example, we verify the experimental accuracy of the proposed BPCU scheme on a challenging set of initial data, which leads to a
	globally smooth solution with very small minima of both $\rho$ and $p$. We simulate the propagation of a supersonic vortex initially
	centered at the origin and propagating with the velocity $(1,1)^\top$. The initial data, prescribed on the computational domain
	$[-5,5]\times[-5,5]$ subject to periodic boundary conditions, are
	$$
	\begin{aligned}
		&\rho(x,y,0)=(1+\delta T)^\frac{1}{\gamma-1},\quad v_i(x,y,0)=1+\delta v_i,~i=1,2,\quad p(x,y,0)=\rho^\gamma(x,y,0),\\
		&(\delta v_1,\delta v_2)=\frac{\varepsilon}{2\pi}e^\frac{1-x^2-y^2}{2}(-y,x),\quad\delta T=-\frac{(\gamma-1)\varepsilon^2}{8\gamma\pi^2}
		e^{(1-x^2-y^2)},
	\end{aligned}
	$$
	where $\varepsilon=10.0828$ is taken to make $\min(\rho(x,y,0))\approx7.8337\times10^{-15}$ and
	$\min(p(x,y,0))\approx1.7847\times10^{-20}$. 
	
	We compute the solution until time $t=0.05$ by the BPCU scheme on sequence of uniform meshes with $\dx=\dy=1/5$, 1/10, 1/20, 1/40, 1/80, and
	1/160. The obtained $L^1$-errors and corresponding experimental rates of convergence are reported in \Cref{tab1}, demonstrating the
	second-order accuracy as expected. We stress that the BP property of the scheme is crucial in this example. For instance, if we compute the
	solution by Scheme 1 with $\dx=\dy=1/60$, then the computed pressure becomes negative after just one time step.
	\begin{table}[ht!] 
		\centering
		\begin{tabular}{|c|lc|lc|lc|lc|}
			\hline
			\multirow{2}{*}{$\dx=\dy$}&\multicolumn{2}{c|}{$\rho$}&\multicolumn{2}{c|}{$\rho v_1$}&\multicolumn{2}{c|}{$\rho v_2$}&
			\multicolumn{2}{c|}{$E$}\\
			\cline{2-9}
			&Error&\!\!Rate&Error&\!\!Rate&Error&\!\!Rate&Error&\!\!Rate\\
			\hline
			1/5&1.96e-2&\!\!--&4.84e-2&\!\!--&4.74e-2&\!\!--&1.19e-1&\!\!--\\ 
			1/10&6.48e-3&\!\!1.60&1.48e-2&\!\!1.71&1.49e-2&\!\!1.67&3.47e-2&\!\!1.77\\ 
			1/20&2.16e-3&\!\!1.59&4.66e-3&\!\!1.66&4.61e-3&\!\!1.69&1.01e-2&\!\!1.78\\ 
			1/40&6.54e-4&\!\!1.72&1.34e-3&\!\!1.80&1.35e-3&\!\!1.78&3.04e-3&\!\!1.74\\ 
			1/80&1.67e-4&\!\!1.97&3.51e-4&\!\!1.94&3.56e-4&\!\!1.92&8.26e-4&\!\!1.88\\ 
			1/160&3.68e-5&\!\!2.18&8.78e-5&\!\!2.00&8.92e-5&\!\!2.00&2.04e-4&\!\!2.02\\ 
			\hline
		\end{tabular}
		\caption{\sf \Cref{ex41}: $L^1$-errors and corresponding convergence rates for the BPCU scheme.\label{tab1}}
	\end{table}
\end{example}

\begin{example}\label{ex42}
	The goal of the second example, which is a modification of the shock-density wave interaction problem from \cite{SO89}, is to show that the
	``built-in'' anti-diffusion is essential to achieve higher resolution. We consider the initial data
	$$
	(\rho(x,0),v(x,0),p(x,0))=\left\{
	\begin{aligned}
		&(3.857143,-0.920279,10.33333),&&x<0,\\
		&(1+0.2\sin(5x),-3.549648,1),&&x>0,
	\end{aligned}
	\right.
	$$
	prescribed on the computational domain $[-10,10]$ subject to free boundary conditions.
	
	We compute the solutions by the proposed BPCU scheme and by Schemes 2 and 3 until time $t=0.2$ on a uniform mesh with $\dx=1/60$. The
	obtained densities are presented in \Cref{fig3}, where one can clearly see that the BPCU scheme outperforms its counterparts.
	\begin{figure}[ht!] 
		\centerline{\includegraphics[width=0.45\textwidth]{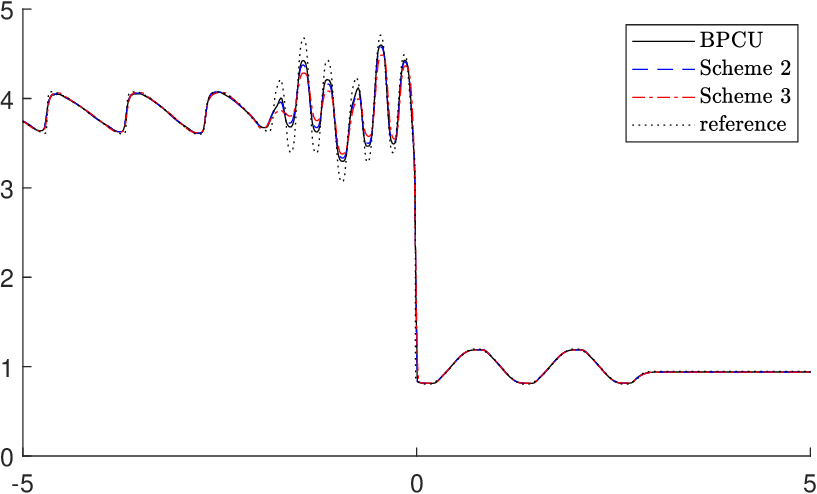}\hspace*{1cm}
			\includegraphics[width=0.45\textwidth]{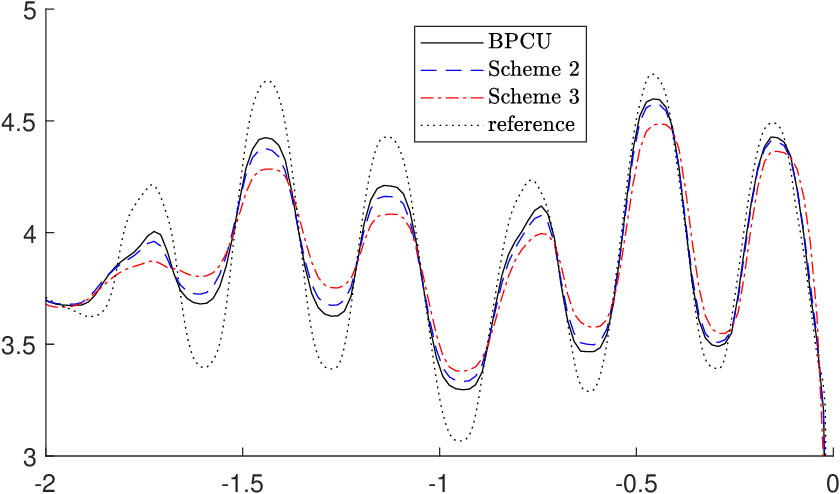}}
		\caption{\sf\Cref{ex42}: Density ($\rho$) computed by the BPCU scheme and Schemes 2 and 3 (left) and zoom at the smooth part of $\rho$
			(right).\label{fig3}}
	\end{figure}
\end{example}

\begin{example}\label{ex43}
	We consider the ``123 problem'' from \cite{EinfeldtMunzRoe1991}: this is a Riemann problem, whose exact solution consists of two rarefaction
	waves and a near-vacuum region between them. The initial data,
	$$
	(\rho(x,0),u(x,0),p(x,0))=\left
	\{\begin{aligned}
		&(1,-2,0.15)&&x<0.5,\\
		&(1,2,0.15),&&x>0.5,
	\end{aligned}
	\right.
	$$
	are prescribed on the computational domain $[0,1]$ subject to free boundary conditions.
	
	We compute the solution by the proposed BPCU scheme until time $t=0.15$ on a uniform mesh with $\dx=1/200$. The obtained density and
	pressure are depicted in \Cref{fig4}, which clearly shows two rarefaction waves and the near vacuum region in the middle. We also plot the
	minimum of both density and pressure as functions of time in \Cref{fig4}: these graphs demonstrates that the positivity of density and
	pressure is preserved by the  BPCU scheme. On the other hand, Scheme 1 violates the BP property at the second stage of first Runge-Kutta
	time step by producing a negative pressure value $p\approx-2.43\times10^{-2}$.
	\begin{figure}[ht!] 
		\centerline{\begin{subfigure}[b]{0.46\textwidth}\caption{density at $t=0.15$}
				\centerline{\includegraphics[width=\textwidth]{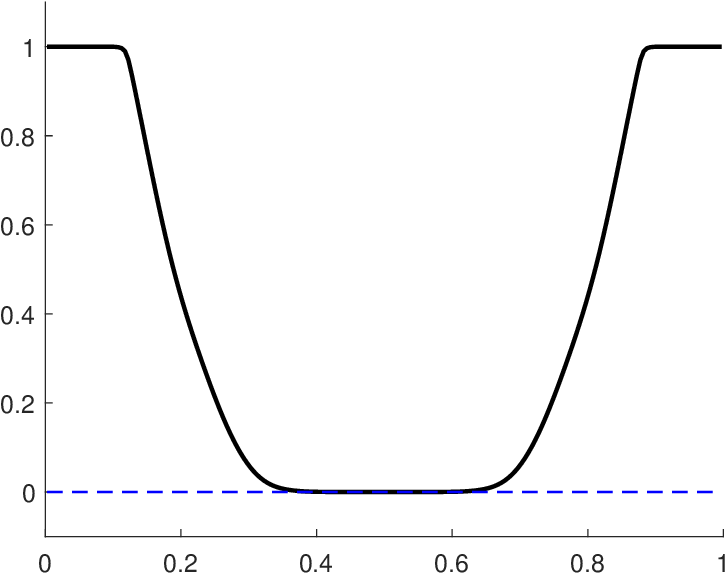}}\end{subfigure}\hspace*{0.2cm}
			\begin{subfigure}[b]{0.46\textwidth}\caption{pressure at $t=0.15$}
				\centerline{\includegraphics[width=\textwidth]{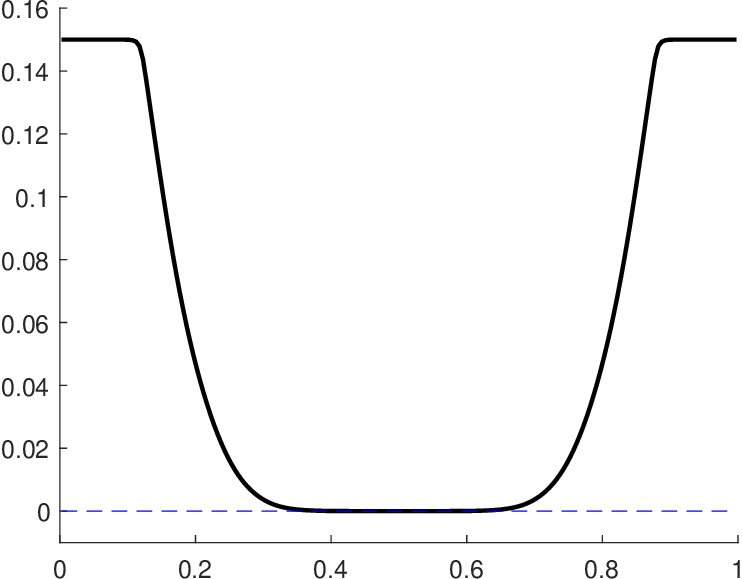}}\end{subfigure}}
		\centerline{\begin{subfigure}[b]{0.46\textwidth}\caption{minimum density as a function of time}
				\centerline{\includegraphics[width=\textwidth]{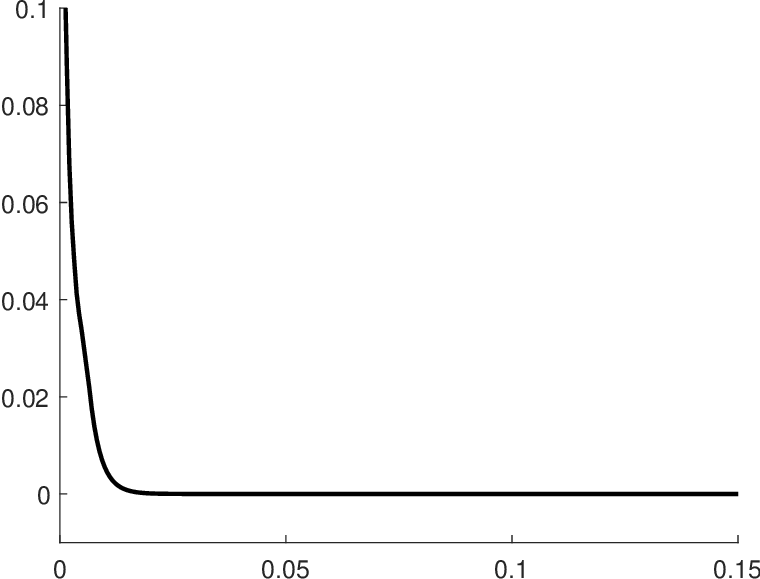}}\end{subfigure}\hspace*{0.2cm}
			\begin{subfigure}[b]{0.46\textwidth}\caption{minimum pressure as a function of time}
				\centerline{\includegraphics[width=\textwidth]{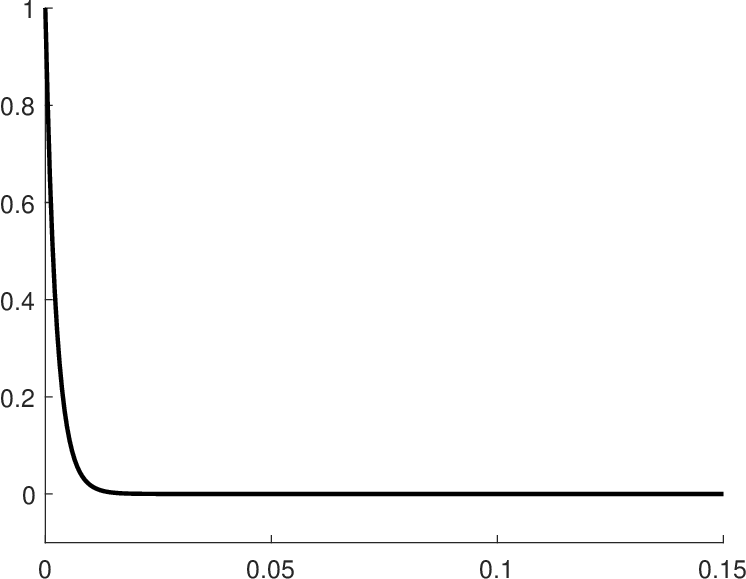}}\end{subfigure}}
		\caption{\sf\Cref{ex43}: Density ($\rho$) and pressure ($p$) computed by the BPCU scheme.\label{fig4}}
	\end{figure}
\end{example}

\begin{example}\label{ex44}
	In this example taken from \cite{ha2005numerical,ZhangShu2010_PP}, we simulate a Mach 80 jet. Initially, the computational domain
	$[0,2]\times[-0.5,0.5]$ is filled with the ambient gas with $\rho(x,y,0)\equiv5$, $v_1(x,y,0)=v_2(x,y,0)\equiv0$, and
	$p(x,y,0)\equiv0.4127$. From the left boundary between $y=-0.05$ and 0.05, a jet of state $(\rho,v_1,v_2,p)=(5,30,0,0.4127)$ is injected into the computational domain, and free boundary conditions are imposed on all other boundaries.
	
	We compute the solution by the BPCU scheme until time $t=0.07$ on a uniform mesh containing $448\times224$ cells. The snapshots of the
	numerical solution at times $t=0.05$ and 0.07, presented in \Cref{fig5}, agree well with those reported in \cite{ZhangShu2010_PP}. The BP
	property of the numerical scheme is crucial in this example: if Scheme 1 is used in this example, negative pressure will appear in the
	numerical result at $t\approx5.277\times10^{-4}$.
	\begin{figure}[ht!] 
		\centerline{\includegraphics[width=0.5\textwidth]{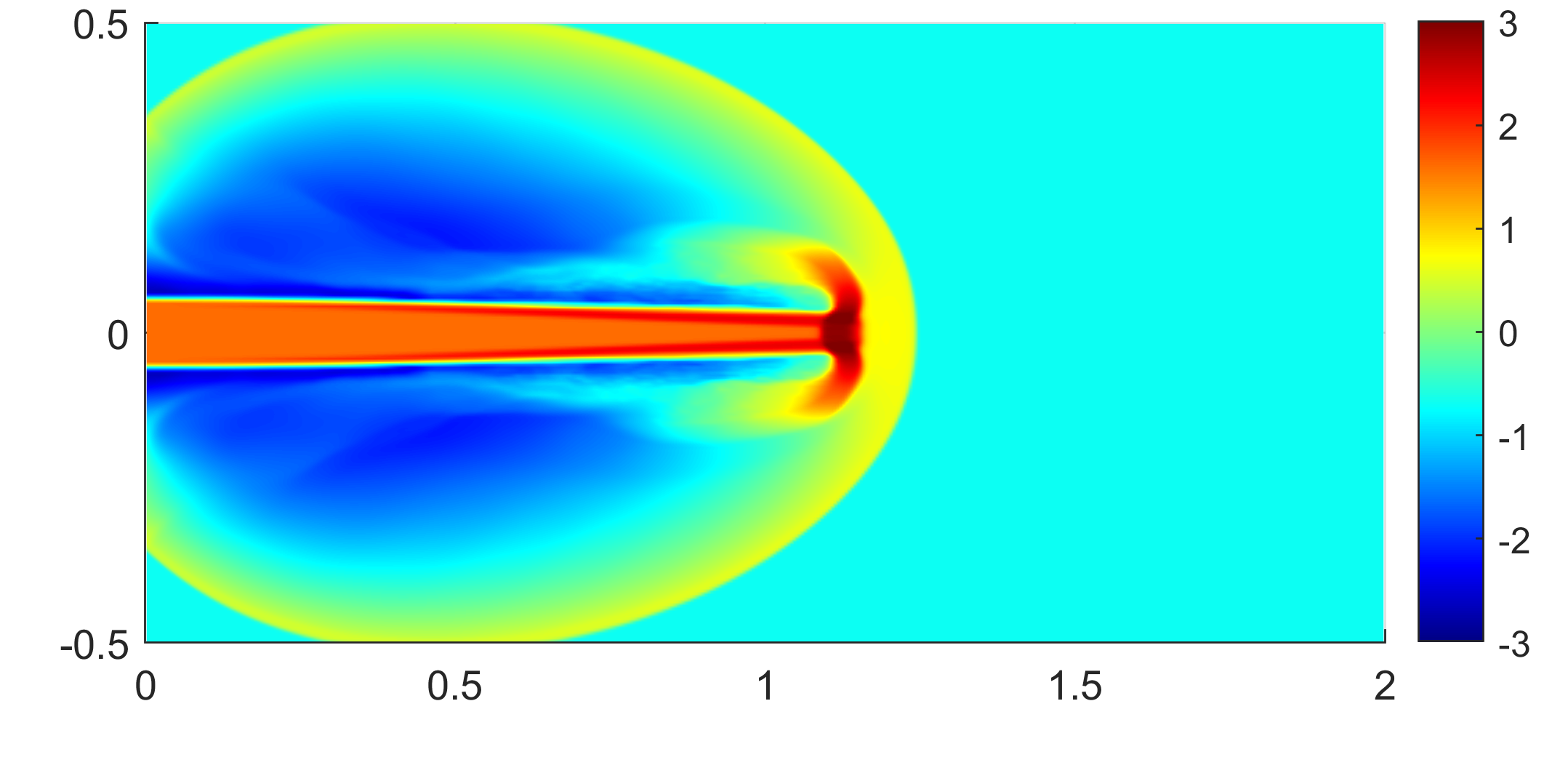}\hspace*{0.2cm}
			\includegraphics[width=0.5\textwidth]{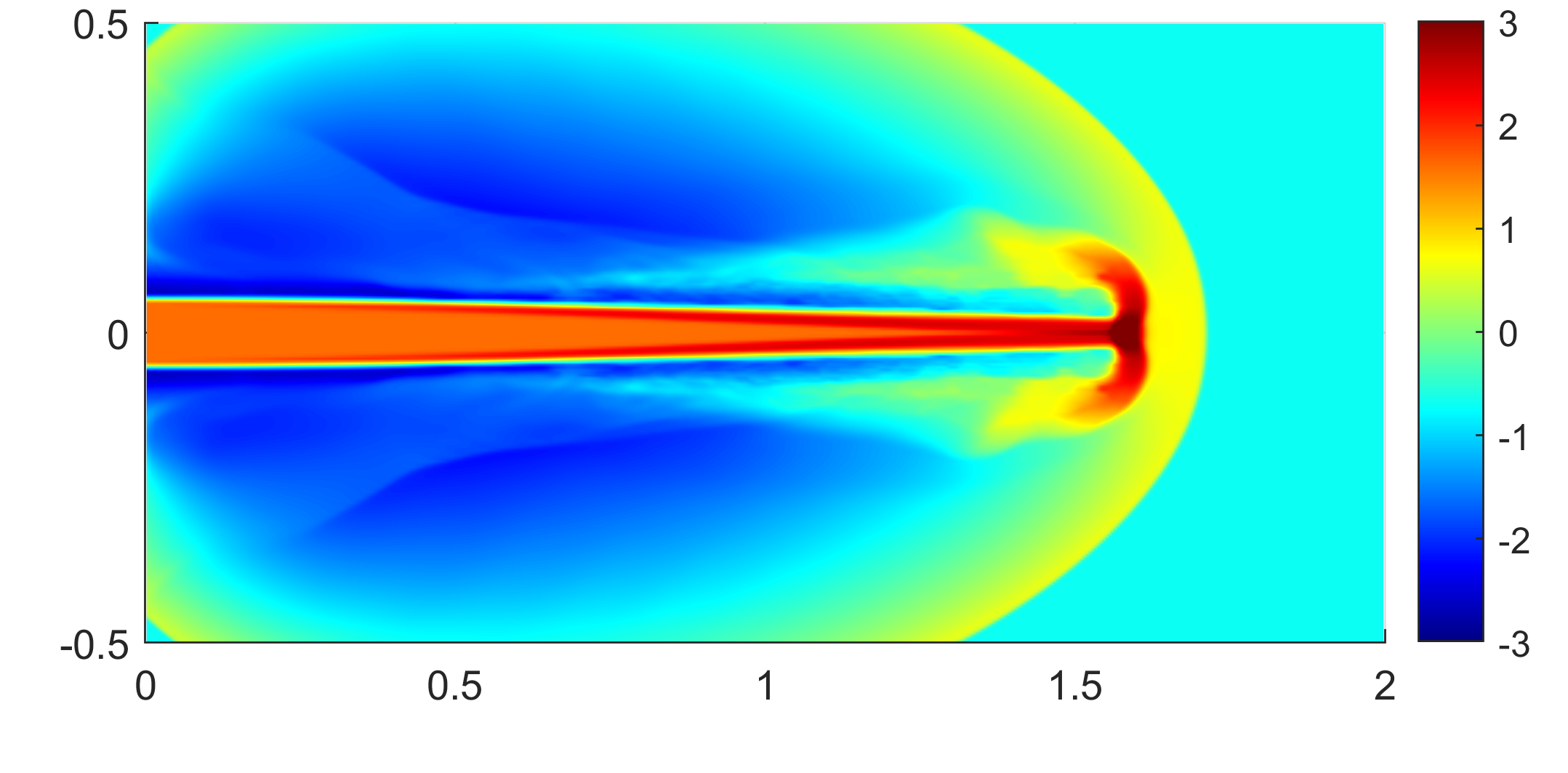}}
		\centerline{\includegraphics[width=0.5\textwidth]{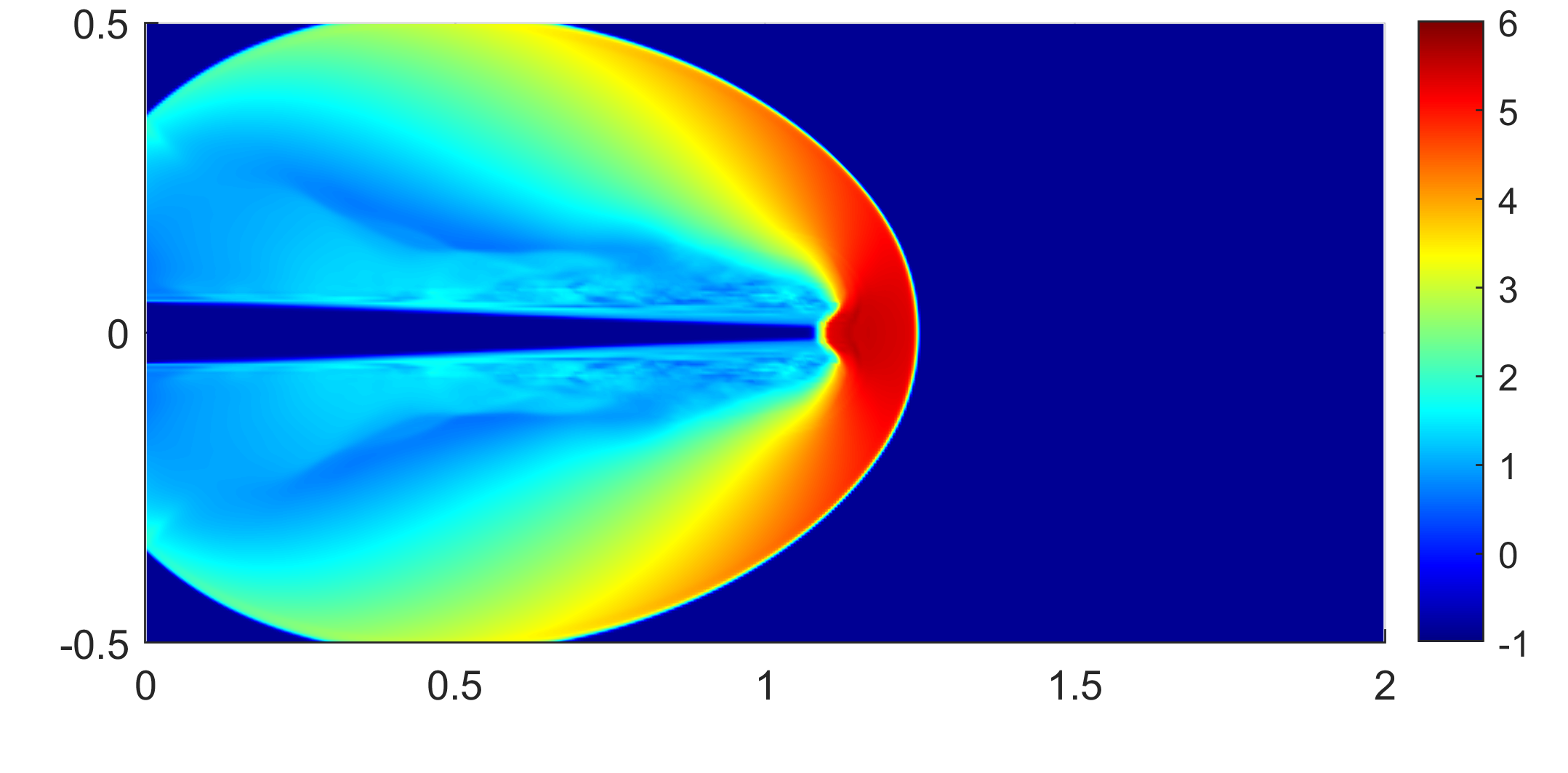}\hspace*{0.2cm}
			\includegraphics[width=0.5\textwidth]{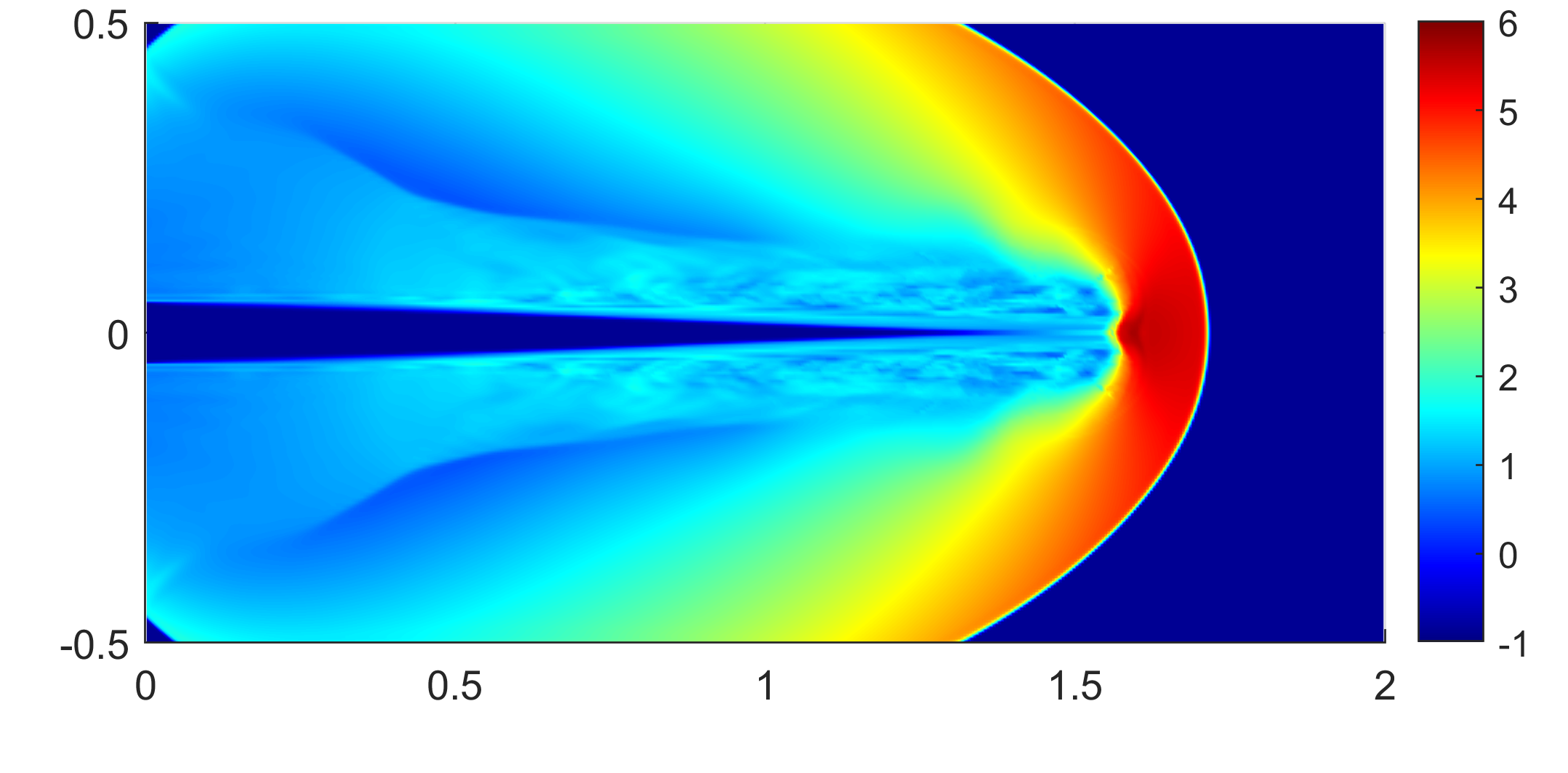}}
		\caption{\sf\Cref{ex44}: Solution ($\ln\rho$ (top row) and $\ln p$ (bottom row)) computed by the BPCU scheme at times $t=0.05$ (left
			column) and 0.07 (right column).\label{fig5}}
	\end{figure}
\end{example}

\begin{example}\label{ex45}
	In this example taken from \cite{ZhangShu2010_PP}, a Mach 2000 jet is considered. The computational domain is $[0,1]\times[-0.25,0.25]$ and
	the initial and boundary conditions are the same as in \Cref{ex44}, except for the jet state which is now
	$(\rho,v_1,v_2,p)=(5,800,0,0.4127)$.
	
	We compute the solution by the BPCU scheme until time $t=0.0015$ on a uniform mesh containing $640\times320$ cells. The snapshots of the
	numerical solution at times $t=0.001$ and 0.0015, presented in \Cref{fig6}, agree well with those reported in \cite{ZhangShu2010_PP}. As in
	\Cref{ex44}, if the BP property is not enforced, that is, if Scheme 1 is used, negative pressure will appear in the numerical solution at
	$t\approx7.15\times10^{-4}$.
	\begin{figure}[ht!]
		\centerline{\includegraphics[width=0.5\textwidth]{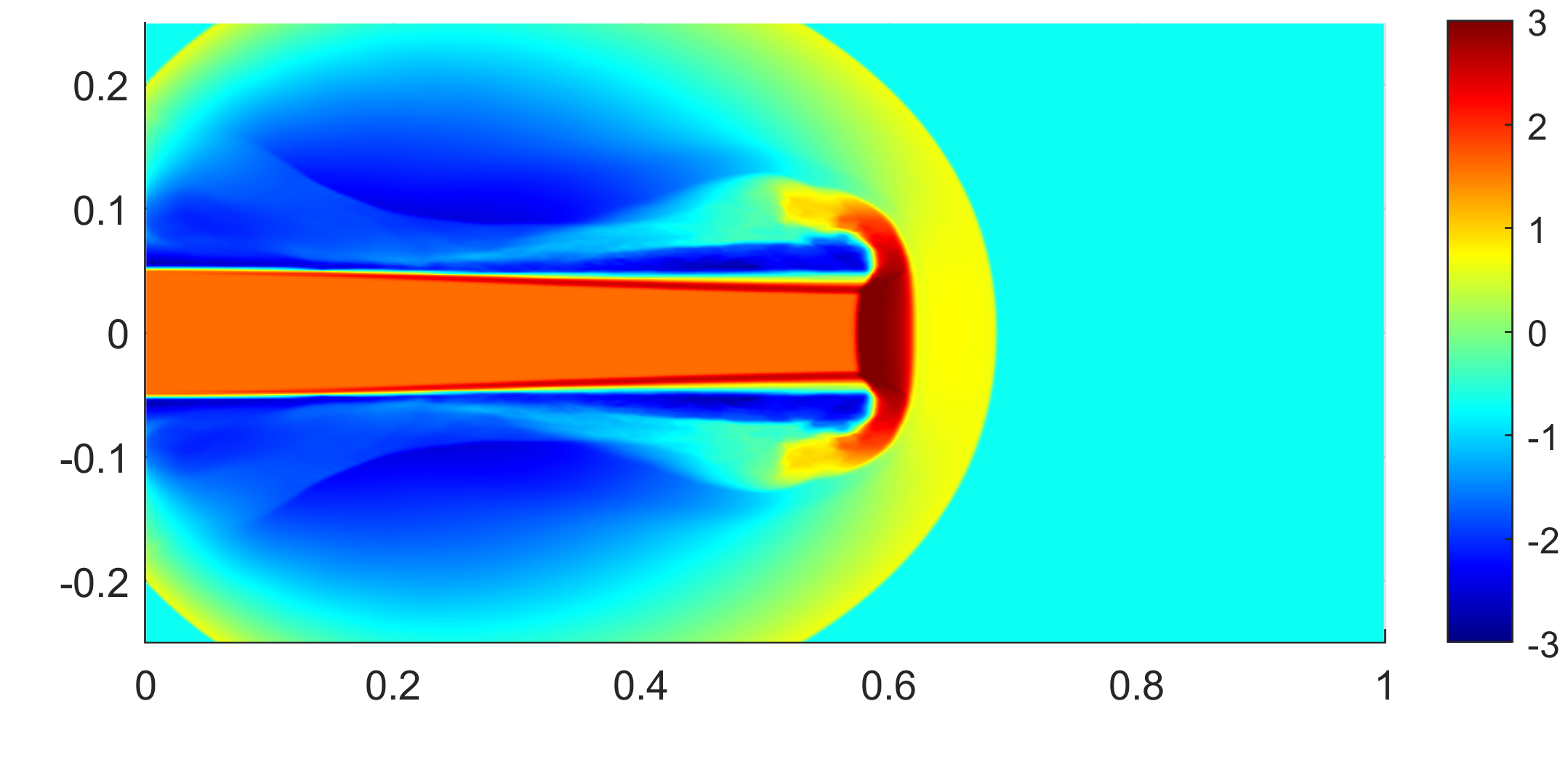}\hspace*{0.2cm}
			\includegraphics[width=0.5\textwidth]{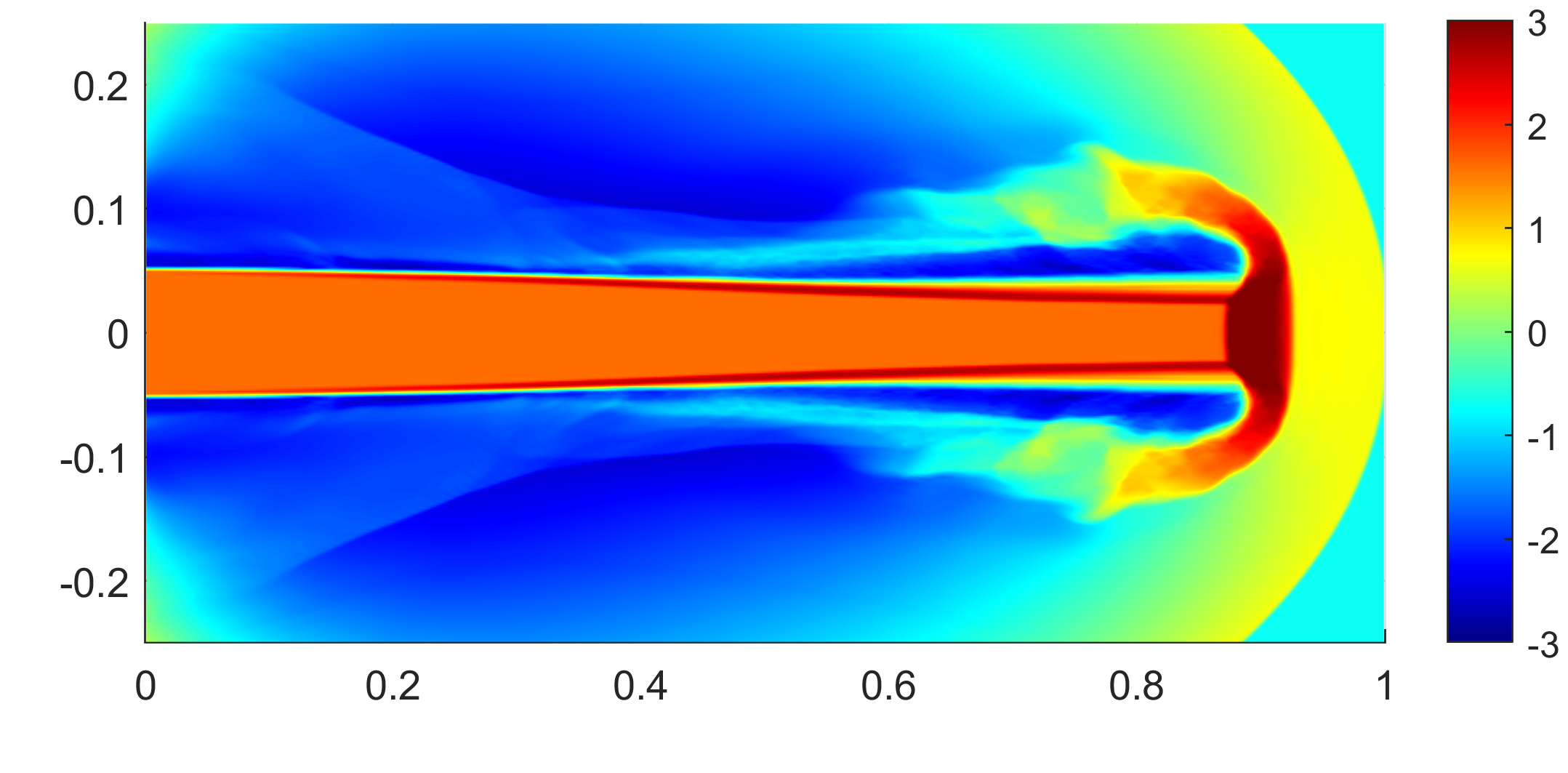}}
		\centerline{\includegraphics[width=0.5\textwidth]{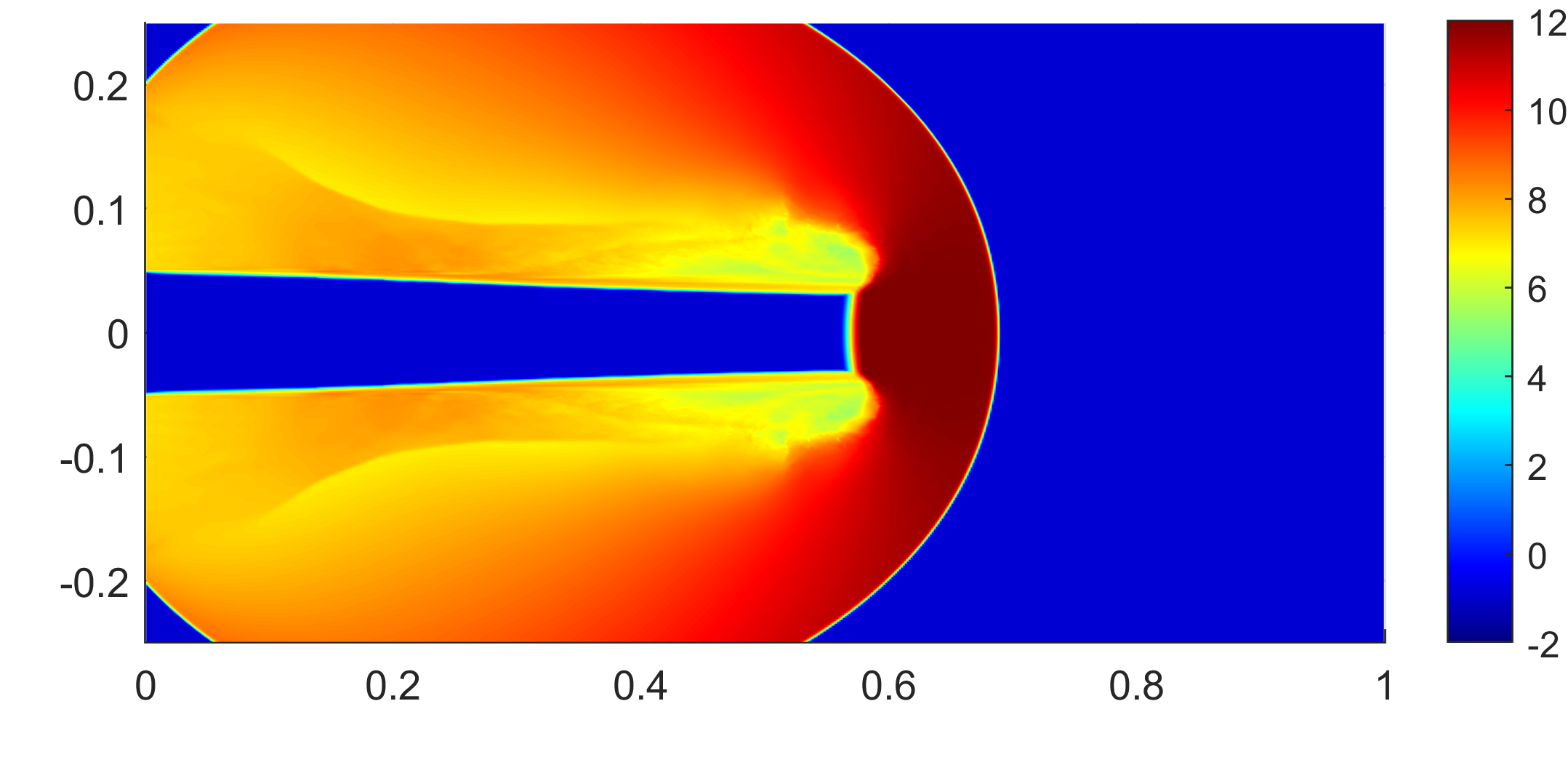}\hspace*{0.2cm}
			\includegraphics[width=0.5\textwidth]{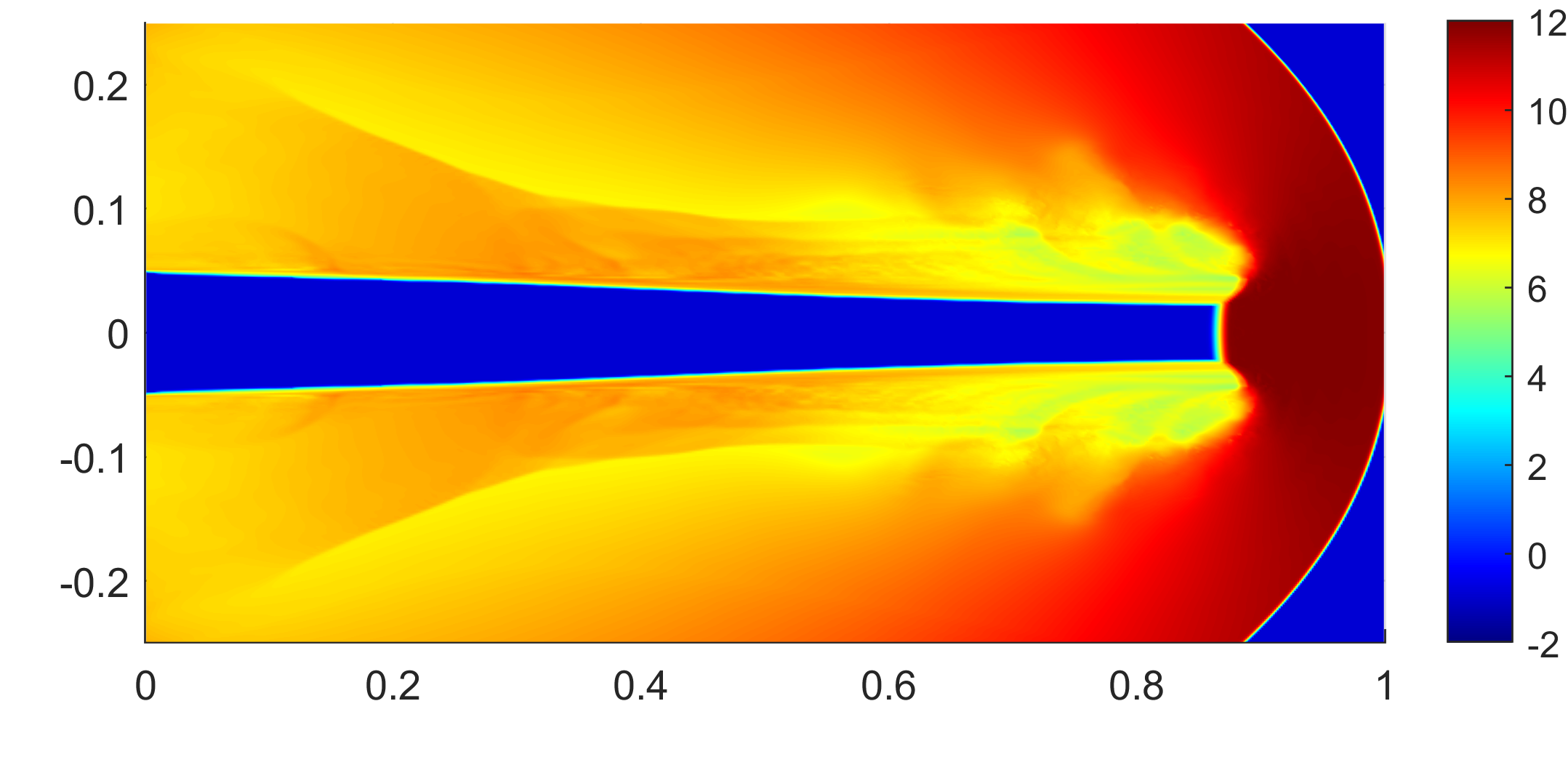}}
		\caption{\sf\Cref{ex45}: Solution ($\ln\rho$ (top row) and $\ln p$ (bottom row)) computed by the BPCU scheme at times $t=0.001$ (left
			column) and 0.0015 (right column).\label{fig6}}
	\end{figure}
\end{example}

\begin{example}\label{ex46}
	This is a shock diffraction benchmark taken from \cite{Quirk1994,CockburnShu1998,ZhangShu2010_PP}. The computational domain is
	$[0,1]\times[6,11]\cup[1,13]\times[0,11]$. At time $t=0$, a pure right-moving shock of Mach number 5.09 is located at $x=0.5$ for
	$y\in[6,11]$. The inflow boundary conditions are imposed at $x=0$ for $y\in[6,11]$, free boundary conditions are used at $x=13$ for
	$y\in[1,11]$, at $y=0$ for $x\in[1,13]$, and at $y=11$ for $x\in[0,13]$, and solid wall boundary conditions are specified at $y=6$ for
	$x\in[0,1]$ and at $x=1$ for $y\in[0,6]$.
	
	Density and pressure computed by the BPCU scheme at time $t=2.3$ on a uniform mesh with $\dx=\dy=1/64$, are presented in
	\Cref{fig7}. The results agree well to those reported in \cite{Quirk1994,CockburnShu1998,ZhangShu2010_PP}. We emphasize that this example is
	challenging due to the very low density and pressure. Numerical schemes that do not satisfy the BP property may generate negative pressure
	and/or density values, leading to failures of computation. For instance, Scheme 1 produces negative pressure at $t\approx0.1074$.
	\begin{figure}[ht!] 
		\centerline{\includegraphics[width=0.5\textwidth,trim = 20 10 35 10,clip]{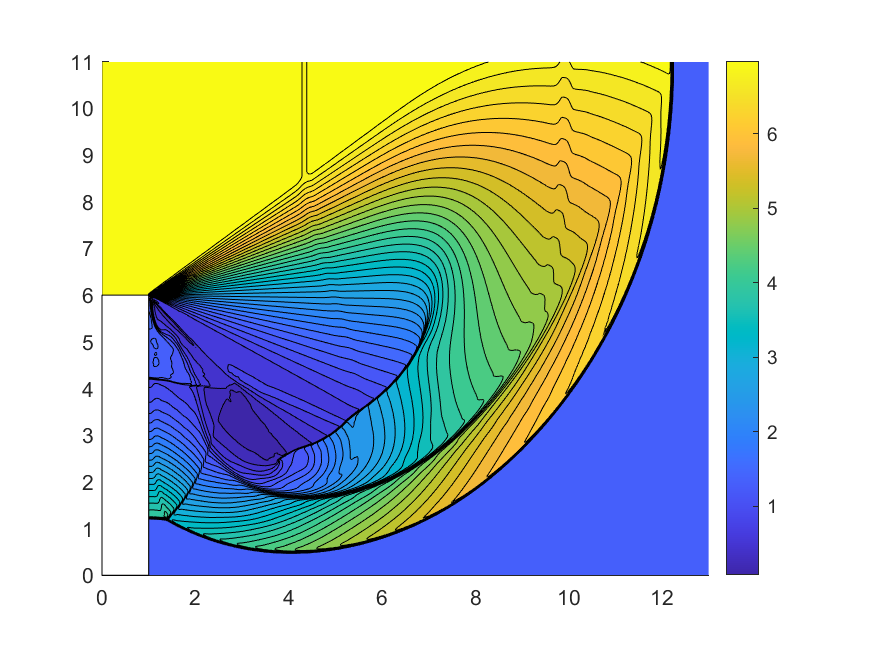}\hspace*{0.2cm}
			\includegraphics[width=0.5\textwidth,trim = 20 10 35 10,clip]{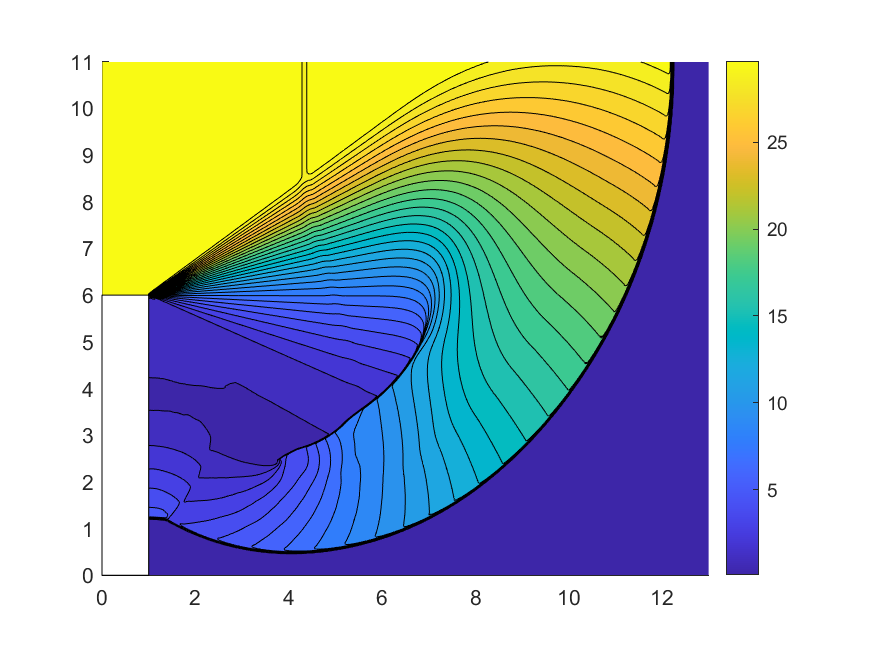}}
		\caption{\sf \Cref{ex46}: Density ($\rho$): 40 equally spaced contour lines from 0.066227 to 7.1568 (left) and pressure ($p$): 40 equally
			spaced contour lines from 0.091 to 37.3 (right) computed by the BPCU scheme.\label{fig7}}
	\end{figure}
\end{example}

\begin{example}\label{ex47}
	The last example is the benchmark from \cite{WC} of flow past a forward facing step. Initially, a right-going Mach 3 uniform flow enters a
	wind tunnel $[0,1]\times[0,3]$, initially filled with a gas with $\rho(x,y,0)\equiv1.4$, $u(x,y,0)\equiv3$, $v(x,y,0)\equiv0$, and
	$p(x,y,0)\equiv1$. The step of height 0.2 is located at $x=0.6$. Solid wall boundary conditions are applied along the walls of the tunnel.
	The inflow and free boundary conditions are applied at the entrance ($x=0$) and the exit ($x=3$), respectively.
	
	We conduct the simulation on two uniform meshes: a coarser (with $\dx=\dy=1/160$) and a finer ($\dx=\dy=1/360$) ones. The density computed
	by the BPCU scheme as well as Scemes 2 and 3 at time $t=4$ are shown in \Cref{fig8}. As one can see, the BPCU scheme achieves higher
	resolution than its counterparts, especially for the slip line issued from the triple point near $(0.6,0.75)$.
	\begin{figure}[ht!] 
		\begin{subfigure}{0.48\textwidth}\caption{$\dx=\dy=1/160$, BPCU scheme}
			\includegraphics[width=\textwidth]{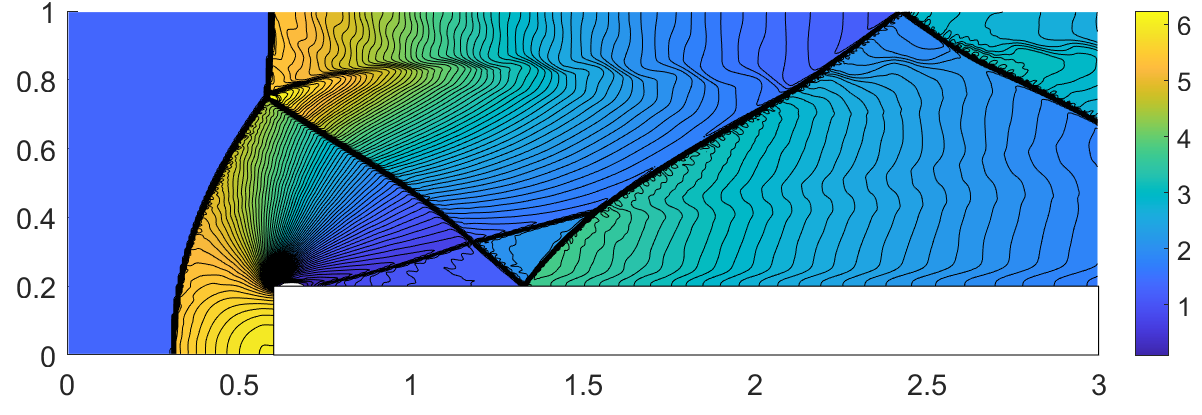}\end{subfigure}\hfill
		\begin{subfigure}{0.48\textwidth}\caption{$\dx=\dy=1/320$, BPCU scheme}
			\includegraphics[width=\textwidth]{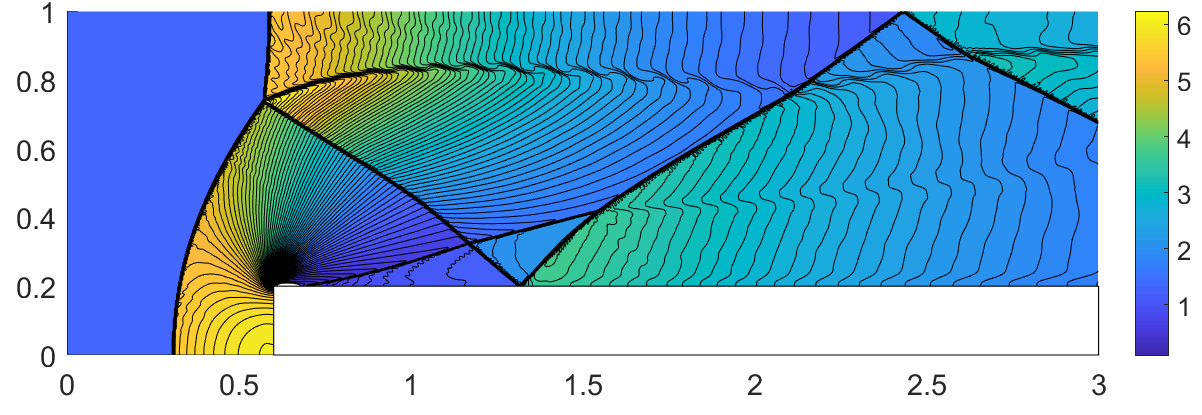}\end{subfigure}
		\begin{subfigure}{0.48\textwidth}\caption{$\dx=\dy=1/160$, Scheme 2}
			\includegraphics[width=\textwidth]{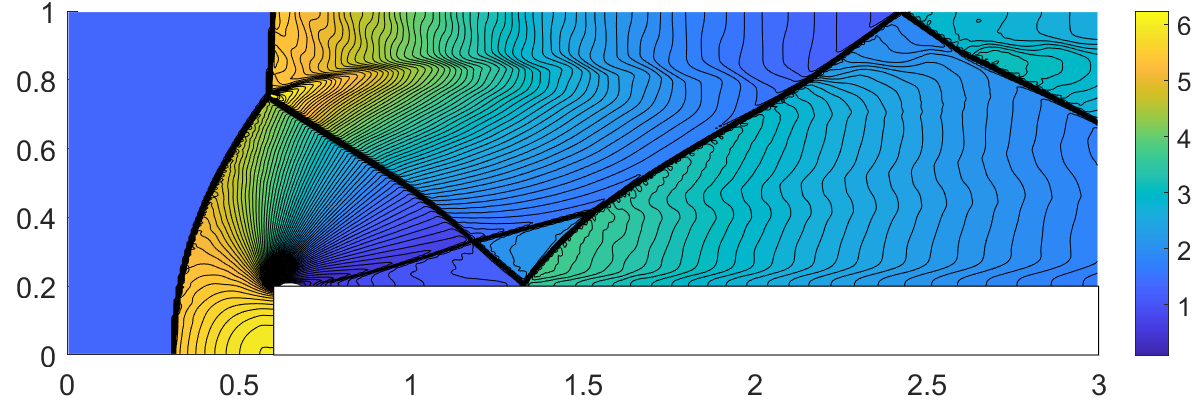}\end{subfigure}\hfill
		\begin{subfigure}{0.48\textwidth}\caption{$\dx=\dy=1/320$, Scheme 2}
			\includegraphics[width=\textwidth]{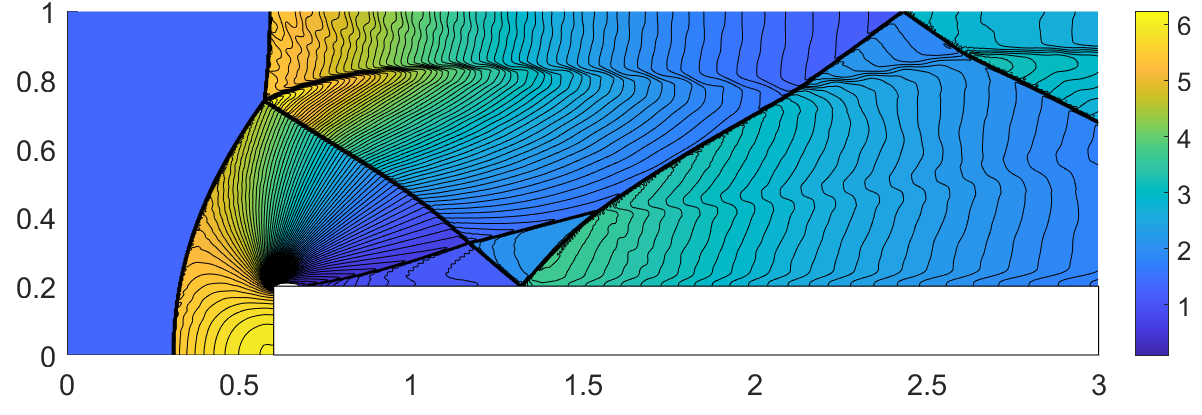}\end{subfigure}
		\begin{subfigure}{0.48\textwidth}\caption{$\dx=\dy=1/160$, Scheme 3}
			\includegraphics[width=\textwidth]{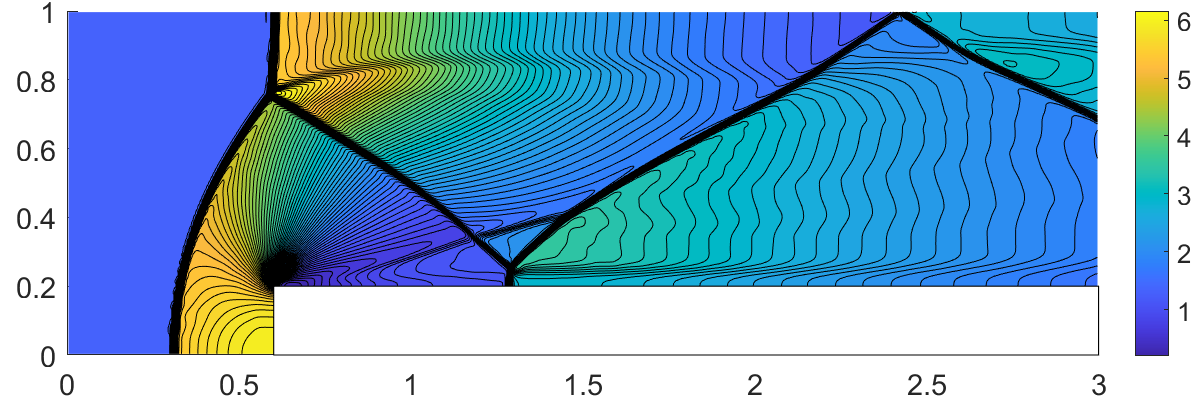}\end{subfigure}\hfill
		\begin{subfigure}{0.48\textwidth}\caption{$\dx=\dy=1/320$, Scheme 3}
			\includegraphics[width=\textwidth]{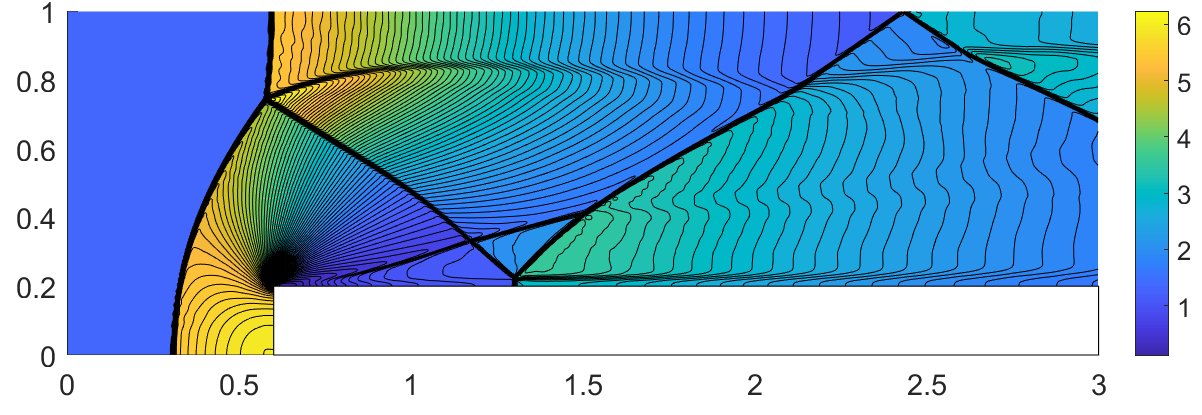}\end{subfigure}
		\caption{\sf Example \ref{ex47}: Density ($\rho$) computed by the BPCU scheme (top row), Scheme 2 (middle row), and Scheme 3 (bottom row) on
			the coarse (left column) and fine (right column) meshes. 80 equally spaced contour lines from $0.090338$ to $6.2365$.\label{fig8}}
	\end{figure}
\end{example}

\section{Conclusions}\label{sec5}
In this paper, we have established a novel framework to analyze and develop bound-preserving (BP) central-upwind (CU) schemes for general
hyperbolic systems of conservation laws. As the foundation of this framework, we have discovered that the CU schemes can be decomposed as a
convex combination of several intermediate solution states. Thanks to this key finding, the goal of designing BPCU schemes was simplified to
the enforcement of four more accessible BP conditions, each of which can be achieved by minor modification to the CU schemes. Our framework
is applicable to general hyperbolic systems of conservation laws and as a specific application, we have employed it to construct provably
BPCU schemes for one- and two-dimensional Euler equations of gas dynamics. The robustness and effectivenss of the proposed BPCU schemes have
been validated by several challenging numerical examples. 

\bibliographystyle{amsplain}

\bibliography{refs}


\end{document}